\documentclass[11pt]{article}
\usepackage[margin=2.5cm,bottom=1.8cm,includefoot]{geometry}
\usepackage{amsmath,amssymb,amsthm}
\usepackage{bbm}
\usepackage{array}
\usepackage{mathrsfs}
\usepackage{enumerate}
\usepackage{graphicx}
\usepackage[utf8]{inputenc}
\usepackage{hyperref}
\usepackage[all]{hypcap}
\usepackage{float}

\usepackage{color}

\newtheorem{thm}{Theorem}
\newtheorem{lem}{Lemma}[section]
\newtheorem{prop}[thm]{Proposition}

\theoremstyle{remark}

\let\BFseries\bfseries\def\bfseries{\BFseries\mathversion{bold}}

\newcommand{\N}{\mathbb{N}}
\newcommand{\R}{\mathbb{R}}
\newcommand{\C}{\mathbb{C}}
\newcommand{\E}{\mathbb{E}}

\renewcommand{\P}{\mathbb{P}}

\newcommand{\CC}{\mathcal{C}}

\newcommand{\BB}{\mathcal{B}}

\newcommand{\FF}{\mathcal{F}}
\newcommand{\1}{\mathbbm{1}}

\newcommand{\epsi}{\varepsilon}

\newcommand{\dd}{\mathrm{d}}

\newcommand{\abs}{|}

\newcommand{\ind}{1\hspace{-0.098cm}\mathrm{l}}

\def\deq{\mathrel{\stackrel{d}{=}}} 

\title{\huge\textbf{Brownian motion with drift conditioned\\to have restricted $L^2$-norm}}
\author{Frank Aurzada\footnote{Technical University of Darmstadt, Darmstadt, Germany, aurzada@mathematik.tu-darmstadt.de, wiegand@mathematik.tu-darmstadt.de},\quad Yuvraj Dutta\footnote{Indian Statistical Institute, Kolkata, India, yuvrajdutta05@gmail.com},\quad \addtocounter{footnote}{-1} Max Wiegand\footnotemark[\value{footnote}]}

\begin{document}

\maketitle

\allowdisplaybreaks

\begin{abstract}
We study the long-time behaviour of Brownian motion with drift under a rare-event conditioning of the quadratic additive functional $Z_T=\int_0^T W_s^2ds$, where $(W_t)_{t\geq 0}$ is a Brownian motion with drift $\mu$. More precisely, we consider the conditional law of $W$ given the event $Z_T\le \theta T$, as $T\to\infty$, allowing the drift $\mu$ to depend on $T$. We derive sharp, uniform small-deviation asymptotics for $Z_T$, including the exact prefactor, and use them to show that the conditioned process converges weakly to an Ornstein--Uhlenbeck process. Remarkably, the limiting dynamics are independent of the original drift.\end{abstract}



\section{Introduction and main results}\label{sec:introandresults}
Conditioning Brownian motion on highly improbable events is a classical topic in probability, going back to the pioneering work of Doob  \cite{doob57}  on Brownian motion conditioned to be positive.  Doob transforms arise naturally when conditioning Markov processes to stay positive, avoid or hit prescribed sets, survive for a long time, or attain certain endpoints. Such ideas have been developed in a wide variety of models, including random walks, L\'evy processes, diffusions, branching processes, and killed Markov processes.

Similarly classical is the work of Knight \cite{KnightBM} on Brownian motion not leaving a fixed interval. This result can be interpreted as conditioning Brownian motion on having a small $L^\infty$-norm. It is therefore natural to ask about conditioning on a small $L^p$-norm; a particularly tractable case here is the $L^2$-norm. This was studied in \cite{bmconditionedtorestrictedl2norm}, where it was shown that Brownian motion conditioned on having small $L^2$-norm converges to an Ornstein--Uhlenbeck process.

The purpose of this paper is to generalize \cite{bmconditionedtorestrictedl2norm}, which considers standard Brownian motion, and study Brownian motion with drift, which is even allowed to depend on $T$. To describe our results, let us set up the required notation: Throughout the paper, we let $(W_t)_{\,t\ge0}$ denote a Brownian motion with drift $\mu = \mu(T)$ starting at $x$, i.e., if $(B_t)_{\,t\ge0}$ is a standard Brownian motion then $W_t := B_t+\mu t+x$, where $x,\mu\in\mathbb R$. Note that $\mu=\mu(T)$ can be a function of $T$ here.
Furthermore, for $T>0$, let
\[
Z_T
:=
\int_0^T W_s^2\,ds
=
\int_0^T (B_s+\mu s+x)^2\,ds.
\]
We investigate the asymptotic behaviour as $T\to\infty$ of the conditional laws
\[
\mathbb P_x\!\left(
(W_t)_{t\ge0}\in\cdot
\,\middle|\,
Z_T\le \theta T
\right),  
\]
for a fixed $\theta>0$. Surprisingly, although the initial Brownian motion has drift $\mu$, the limiting conditioned process turns out to be the same Ornstein--Uhlenbeck process as in the driftless case. Thus the effect of the original drift disappears in the limit under the conditioning. An interesting question is how strong the drift (as a function of $T$) may be so that this result prevails.

More generally speaking, conditioning additive functionals on atypical values is a natural problem in the study of stochastic processes and rare events. Here we consider the quadratic additive functional $Z_T=\int_0^T W_s^2\,ds$ and ask for the effective long-time dynamics of Brownian motion with drift under an atypically small value of $Z_T$.

Our main result is the following.


\begin{thm}\label{thm:conditionallimit}
Fix $x\in \R$ and $\theta>0$. Suppose that $\mu^2 \leq M T^\beta$ for some $M>0$ and \(\beta<\frac{1}{2}\) and all sufficiently large $T$. Then, as $T\to\infty$, the probability measures
\[
\mathbb P_x\!\left(
(W_t)_{t\ge0}\in\cdot
\,\middle|\,
Z_T\le\theta T
\right)
\]
converge weakly on $C([0,\infty))$ to the law of the Ornstein--Uhlenbeck process
$(X_t)_{t\ge0}$ satisfying
\begin{equation} \label{eqn:defnou}
X_0=x,
\qquad
dX_t
=
dB_t
-\frac{1}{2\theta}X_t\,dt,
\qquad t\ge0,
\end{equation}
where \((B_t)_{t\ge0}\) is a standard Brownian motion.
\end{thm}
Let us discuss the nature of the conditioning $Z_T\leq \theta T$. For this purpose, notice that
\[
\mathbb E_x[Z_T]
=
\frac{\mu^2}{3}\, T^3
+\left(\mu x+\frac12\right)T^2
+x^2T.
\]
Therefore, the typical scale of $Z_T$ is order $\mu^2 T^3$ or $T^2$ depending on $\mu$. Thus, the event we condition on is a scenario where $Z_T$ is supposed to be much smaller than its mean, i.e.\ we deal with a small-deviation regime. Another way of looking at this is by observing that
$$
Z_T \deq T^2 \int_0^1 \left( B_s + \mu \sqrt{T} s + \frac{x}{\sqrt{T}}\right)^2 \dd s
$$
and hence $Z_T \leq \theta T$ is the same as
\begin{align} \label{eqn:smallballproblem}
\left\lVert B + \mu \sqrt{T} \cdot + \frac{x}{\sqrt{T}} \right\rVert_{L^2[0,1]} \leq \sqrt{\frac{\theta}{T}},
\end{align}
so that the event we condition on is a non-centred small ball problem with shrinking radius and at the same time $T$-dependent centre.
Therefore,
the proof of Theorem~\ref{thm:conditionallimit}
will rely on precise asymptotics for the small ball probabilities of $Z_T$.
These asymptotics are of independent interest and constitute one of the main technical contributions of the paper.

\begin{thm}\label{thm:asymptotics}
Let $\theta>0$ be fixed. Let $x= x(T)$, $y=y(T)$, and $\mu=\mu(T)$ be functions of $T$ and suppose that
\[
\mu^2+x^2\le MT^\beta,
\qquad
|y-\theta T|\le KT^\beta,
\]
for some $M,K>0$ and \(\beta<\frac{1}{2}\) and sufficiently large $T$. Then, as $T\to\infty$,
\[
\mathbb P_x(Z_T\le y)
\sim
\frac{4\sqrt y}{\sqrt{\pi}\,T}
\exp\!\left(
-\frac{(T+x^2)^2}{8y}
-\frac{\mu^2T}{2}
-\mu x
+\mu^2\frac yT
\right).
\]
\end{thm}

A surprising consequence is that, although the small ball probabilities depend explicitly (and significantly) on $\mu$, the weak limit of the conditioned process does not. The conditioning completely suppresses the effect of the original drift in the limit. We discuss this feature in detail in Subsection~\ref{sec:driftdiscussion} below. \\

We also establish the following extension of the corresponding result for $\mu=0$ in \cite{bmconditionedtorestrictedl2norm}. Here the deterministic constraint on the quadratic functional is replaced by an independent exponential random threshold. This randomized formulation connects the conditioning problem directly to the classical theory of killed Markov processes: it is equivalent to conditioning the diffusion on survival up to time $T$ under quadratic killing.

\begin{thm}\label{thm:exponential}
Fix $x\in \R$. Let $\gamma>0$ and $\eta$ be an exponential random variable with mean $\frac{2}{\gamma^2}$, independent of $W$. Suppose $|\mu| e^{-\gamma T}\rightarrow0 $ as $T\to \infty$. Then, as \(T\to\infty\), the probability measures
\[
\mathbb P_x\!\left((W_t)_{t\ge0}\in\cdot\,\middle|\, Z_T\le\eta\right)
\]
converge weakly on \(C([0,\infty))\) to the law of the Ornstein--Uhlenbeck process
\((Y_t)_{t\ge0}\) satisfying
\[
Y_0=x,\qquad
dY_t=dB_t-\gamma Y_t\,dt,\qquad t\ge0,
\]
where \((B_t)_{t\ge0}\) is a standard Brownian motion.
\end{thm}

We stress that despite the fact that the drift is even allowed to increase exponentially up to a certain extent, the limit is not influenced by it, cf.\ the discussion in Subsection~\ref{sec:driftdiscussion} below.

\medskip

The proof of Theorem~\ref{thm:asymptotics} follows the general philosophy also used in \cite{ouconditionedrestricted}. Starting with an explicit formula for the Laplace transform of $Z_T$, we recover its distribution function through the Bromwich inversion formula. The resulting contour integral is then analysed by means of a saddle-point method. A careful decomposition of the contour integral yields asymptotically sharp estimates that are uniform in the parameters $x$, $\mu$ and $y$ in the relevant regime for conditioning. Once the small ball asymptotics have been established, Theorem~\ref{thm:conditionallimit} is obtained from the corresponding asymptotics for ratios of small ball probabilities. Namely, these ratio asymptotics yield the Radon--Nikodym density of the limiting conditioned measure on finite time intervals. We then identify this density as a stochastic exponential and apply Girsanov's theorem to determine the law of the limiting process, which turns out to be an Ornstein--Uhlenbeck process. The proof of Theorem~\ref{thm:exponential} is more direct. Using the memoryless property of $\eta$ and explicit formulas for the corresponding killed semigroup and for the survival probability, we represent the transition density of the conditioned process 
in terms of the (limiting) Ornstein--Uhlenbeck transition density multiplied by a factor that converges to one. Scheff\'e's lemma then yields the $L^1$-convergence of these transition densities and convergence in total variation on finite time frames, which implies weak convergence.

\paragraph{Outline.} The remainder of the paper is organised as follows. In the next subsection, we discuss our assumptions on the $T$-dependent drift $\mu$, further directions for generalisations, and related work. Section~\ref{sec:sbp} is devoted to the proof of the uniform small ball asymptotics of Theorem~\ref{thm:asymptotics} via a saddle-point analysis argument. Section~\ref{sec:unif_ratiosbp} contains the uniform asymptotics for the ratios of these small ball probabilities. These estimates are then used in Section~\ref{sec:proof_main} to identify the limiting change of measure and, via Girsanov's theorem, prove Theorem~\ref{thm:conditionallimit}. Section~\ref{sec:proof_exponential} is devoted to the proof of Theorem~\ref{thm:exponential}. Finally, Appendix ~\ref{sec:auxiliary1} contains several auxiliary identities and inequalities used throughout the paper.

\subsection{Discussion, generalisations, and related work} \label{sec:driftdiscussion}

It is interesting to explore the conditions on the drift that make Theorem~\ref{thm:conditionallimit} and, respectively, Theorem~\ref{thm:exponential} work. Let us start with the former. 
Our heuristic here is as follows. Let us denote by $\mathbb P_x^\mu$ the distribution of Brownian motion $W$ with drift $\mu$. Representing $\mathbb P_x^\mu$ by $\mathbb P_x^0$ with the help of the Cameron-Martin formula gives
\begin{equation} \label{eqn:driftdiscussion}
\mathbb P_x^\mu\!\left(
(W_t)_{t\ge0}\in A
\,\middle|\,
Z_T\le\theta T
\right) = \frac{ \E_x^0[ \ind_{W \in A}\, \ind_{Z_T\leq \theta T}\,\, e^{-\mu x -\mu^2 T/2 + \mu W_T}]}{\E_x^0[ \ind_{Z_T\leq \theta T}\,\, e^{-\mu x -\mu^2 T/2 + \mu W_T}]} =  \frac{ \E_x^0[ \ind_{W \in A}\, \ind_{Z_T\leq \theta T}\,\,e^{\mu W_T}]}{\E_x^0[ \ind_{Z_T\leq \theta T}\,\, e^{ \mu W_T}]}.
\end{equation}
The last formula shows that we can work under the drift-less measure $\P_x^0$ with the tilt $e^{\mu W_T}$. We already know from \cite{bmconditionedtorestrictedl2norm} that, for sets $A\in\mathcal{F}_{t_0}$, i.e.\ depending only on an initial fixed time frame, the conditioned drift-less measure $\mathbb P_x^0$ converges to the law of an Ornstein-Uhlenbeck process $X$ with parameter $1/(2\theta)$,  as in (\ref{eqn:defnou}). Therefore, as a heuristic, the ratio in (\ref{eqn:driftdiscussion}) becomes essentially
$$
\frac{ \E_x[ \ind_{X\in A}\,\, e^{\mu X_T}]}{\E_x[ e^{ \mu X_T}]} =: \P^{\text{tilt}}(X \in A ),
$$
as long as the tilt does not influence the remaining $L^2$-budget $Z_T\leq \theta T$. Since $X$ is a Gaussian process, the effect of the tilt $e^{\mu X_T}$ simply changes the expectation of the process by $\E^{\text{tilt}}[X_t] - \E[X_t] =  \mu \operatorname{cov}(X_t,X_T) \approx \mu \theta e^{-\frac{1}{2\theta}(T-t)}$. Due to the exponential factor, this change in the mean is negligible for most of the time span, but it does become relevant for $t$ close to $T$. Then it also changes the $L^2$-functional. Namely, the resulting shift incurs an additional $L^2$-norm contribution of $\int_0^T  (\mu \theta e^{-\frac{1}{2\theta}(T-t)} )^2\,dt \approx \mu^2$, which has to be accounted towards fulfilling the overall budget condition $Z_T\leq \theta T$. The above heuristic thus only works if the additional $L^2$-norm does not influence the condition $Z_T\leq \theta T$, i.e.\ if $\mu^2 \ll T$. In this case, the tilted measure becomes the untilted measure $\mathbb P_X$ considered in (\ref{eqn:defnou}), at least for sets $A\in\mathcal{F}_{t_0}$ in a fixed time frame. Certainly, this is only a heuristic, but it would be interesting to make it precise, i.e.\ to show Theorem~\ref{thm:conditionallimit} under the assumption $\mu^2 \ll T$. We stress that this would  require a different approach than ours, because our Theorem~\ref{thm:asymptotics} is expected to work only under the condition $\mu^2 \ll T^{1/2}$ (which is roughly what we showed), because otherwise additional terms do appear in the small ball asymptotics.

Let us make a similar discussion for Theorem~\ref{thm:exponential}. The analog of (\ref{eqn:driftdiscussion}) is much simpler in the case of survival under quadratic killing: 
\begin{align*} 
\mathbb P_x^\mu\!\left(
(W_t)_{t\ge0}\in A
\,\middle|\,
Z_T\le \eta
\right) 
& =   \frac{ \E_x^\mu[ \ind_{W \in A}\, e^{-\frac{\gamma^2}{2} Z_T}]}{\E_x^\mu[ e^{-\frac{\gamma^2}{2} Z_T}]}
=
\frac{ \E_x^0[ \ind_{W \in A}\,\,  e^{-\frac{\gamma^2}{2} Z_T} e^{-\mu x -\mu^2 T/2 + \mu W_T}]}{\E_x^0[ e^{-\frac{\gamma^2}{2} Z_T} e^{-\mu x -\mu^2 T/2 + \mu W_T}]} 
\\
& =  \frac{ \E_x^0[ \ind_{W \in A}\,\, e^{-\frac{\gamma^2}{2} Z_T} e^{\mu W_T}]}{\E_x^0[ e^{-\frac{\gamma^2}{2} Z_T} e^{ \mu W_T}]}.
\end{align*}
Now, again we can argue that under the drift-less measure considered in the last term, when $T\to\infty$, we know from \cite{bmconditionedtorestrictedl2norm} that the measure converges to an Ornstein-Uhlenbeck process. However, this time there is no fixed $L^2$-budget to be kept in mind. It only remains to study when the tilt is negligible, i.e.\ does not show up in the limit. As discussed above, the tilt corresponds to a change in the expectation of the limiting Ornstein-Uhlenbeck process of order $\approx \mu e^{-\gamma(T-t)}$, which is negligible for $|\mu| e^{-\gamma T} \to 0$; and this is exactly the sufficient condition in our Theorem~\ref{thm:exponential}.

It would be interesting to study the influence of the $T$-dependent drift further. E.g.\ what happens for the critical scales, i.e.\ $\mu \sim c T^{1/2}$ in Theorem~\ref{thm:conditionallimit} and $\mu\sim c e^{\gamma T}$ in Theorem~\ref{thm:exponential}, and beyond? One could also study endpoint conditionings $W_T=a_T$, which should be particularly suitable to our explicit technique and would clarify the influence of $\mu$ even further. Other more general interesting directions are the restrictions on $L^p$-norms, $p\not\in\{2,\infty\}$, and the same questions for higher dimensions.

\paragraph{Related work.}
We have already mentioned \cite{bmconditionedtorestrictedl2norm}, as this was the starting point of our work. A related paper is \cite{ouconditionedrestricted} where the same question as in \cite{bmconditionedtorestrictedl2norm} is studied when starting with an Ornstein--Uhlenbeck process.
A very recent paper that also aims to generalize \cite{bmconditionedtorestrictedl2norm} is the preprint \cite{papertobiasschmidt}. There, conditionings of more general diffusions on rare events are discussed, which includes the result in \cite{bmconditionedtorestrictedl2norm}. However, a result for Brownian motion with drift, as our Theorem~\ref{thm:conditionallimit}, is not specified and the approach is methodologically completely different from ours: While we use tools from stochastic analysis,  \cite{papertobiasschmidt} takes a statistical-mechanics and spectral-theoretic viewpoint, with tools based on path ensembles, Feynman–Kac semigroups, and analytic perturbation theory.

A paper closely related to our Theorem~\ref{thm:asymptotics} is \cite{LL93} (also see \cite{KuelbsLiLinde94,LindeRosinski94}), where  shifted small ball problems are considered as well. However, these results are not applicable to our precise setup (\ref{eqn:smallballproblem}) and the technique is also different from ours. Similarly, \cite{Fatalov13} proves a version of our Theorem~\ref{thm:asymptotics} for fixed $x$ and $\mu=0$. For further literature on $L^2$ small ball estimates, see, e.g.\ \cite{Li2001}, as well as the surveys on the subject \cite{LiShao01,NazarovPetrova2023}, and references therein.

The distribution of the integrated squared Brownian motion with drift has been studied in detail in \cite{Xia20}, where explicit series representations for its density and distribution function are obtained, together with sharp asymptotics for small and large values of the argument for fixed time horizons. However, the regime studied in our Theorem~\ref{thm:asymptotics} is not included there.

Quadratic killing of Brownian motion was studied in \cite{Wenocur86}, where explicit results for the associated killing-time distribution are obtained, equivalently for the Laplace transform of $\int_0^T W_s^2\,ds$. Our randomized conditioning in Theorem~\ref{thm:exponential} has exactly this interpretation, but our interest lies in the conditional path law given survival and its long-time limit rather than on precise information about the killing time distribution.

On a more conceptual level, Chetrite and Touchette \cite{ChetriteTouchette13,ChetriteTouchette2015} developed a general theory relating hard conditioning of additive functionals (``microcanonical path ensembles'') to exponential tilting (``canonical ensembles'') and studied the associated driven process obtained by a generalized Doob transform. The deterministic conditioning of Theorem~\ref{thm:conditionallimit} and the exponential randomization in Theorem~\ref{thm:exponential} fit naturally into this framework, and we prove weak convergence of the hard-conditioned path law itself and obtain sharp small-deviation asymptotics including the prefactor.

Other interesting and conceptually related papers are \cite{BrycDembo1997,BercuGamboaLavielle2000,duBuissonTouchette2023,RoynetteYor09,RoynetteYor10}.

\section{Small ball probability of \texorpdfstring{$Z_T$}{ZT}}
\label{sec:sbp}

The Laplace transform for $Z_T$ has an explicit form which is given in \cite{borodin2002handbook}. The proof of Theorem~\ref{thm:asymptotics} takes advantage of this using a saddle-point analysis of the Bromwich inversion formula. We first split the inversion integral into a main integral and an error integral in Subsection~\ref{subsec:splitting}. The error integral is estimated uniformly along the Bromwich contour and shown to be negligible compared to the leading asymptotic contribution. These estimates are established in Subsection~\ref{subsec:errorintegral}.
We then localise the main integral around the saddle point. To this end, the Bromwich contour is partitioned into three regions: a shrinking neighbourhood of radius \(T^{-\alpha}\) centred at the limiting saddle point, the annular region between the radii \(T^{-\alpha}\) and \(r\) (for an appropriate $r$ defined later on), and the remaining part of the contour outside the ball of radius \(r\). The contributions from the latter two regions are shown to be negligible, while the asymptotic behaviour is entirely determined by the integral over the shrinking neighbourhood. The corresponding localisation estimates are derived in Subsection~\ref{subsec:localisation}.
In Subsection~\ref{subsec:shrinkingball} we analyse the integral over the shrinking neighbourhood and identify its leading asymptotic behaviour. Finally, in Subsection~\ref{subsec:finalasymptotic}, the resulting saddle-point expression is evaluated explicitly and shown to be equivalent to the form stated in Theorem~\ref{thm:asymptotics}. The argument yields the following version of Theorem~\ref{thm:asymptotics} which is uniform over the admissible range of the parameters.

\begingroup 
\renewcommand{\thethm}{2'}
\begin{thm}\label{thm:uniformasymptotics}
Let $\theta>0$, $\beta<\frac{1}{2}$, $M,K>0$, $\delta\in(0,1)$ be fixed. Then there is a $T_{M,\delta,\beta,K}>0$ such that for all $x=x(T),\mu=\mu(T)$ and $y=y(T)>0$ satisfying
\[
\mu^2+x^2\le MT^\beta,
\qquad
|y-\theta T|\le KT^\beta,
\]
for all $T\geq T_{M,\delta,\beta,K}$, we have
\[
(1-\delta)\,A_T(x,\mu,y)
\le
\P_x(Z_T\leq y)
\le
(1+\delta)\,A_T(x,\mu,y),
\]
for all $T\geq T_{M,\delta,\beta,K}$, where
\[
A_T(x,\mu,y)
:=
 \frac{4\sqrt y}{\sqrt{\pi}\,T}\,
\exp\!\left(
-\frac{(T+x^2)^2}{8y}
-\frac{\mu^2T}{2}
-\mu x
+\mu^2\frac yT
\right).
\]
\end{thm}
\endgroup


The remainder of this section is devoted to the proof of Theorem~\ref{thm:uniformasymptotics}. Throughout the proof we will w.l.o.g. assume that $\beta>0$, as it relaxes the restriction on the admissible parameters without influencing the result.\\
The Laplace transform of \(Z_T\) is known explicitly  \cite[Part~II, Section~2, Formula~(1.9.3)]{borodin2002handbook} and is given by
\[
L_T(u)
:=
\mathbb E\!\left[e^{-uZ_T}\right]
=
\frac{e^{-\mu x}}{\sqrt{\cosh(\sqrt{2u}\,T)}}
\exp\!\left(
-\frac{\mu^2T}{2}
-
\frac{
(2ux^2-\mu^2)\sinh(\sqrt{2u}\,T)
-
2\mu\sqrt{2u}\,x
}{
2\sqrt{2u}\cosh(\sqrt{2u}\,T)
}
\right),
\]
for $\Re u>0$, where the principal branch of the complex square root is used. The principal branch of the square root is analytic on
\(
\mathbb C\setminus(-\infty,0]
\).
Moreover, $\cosh(\sqrt{2u}\,T)=0$ if and only if
$u=-\pi^2(2k+1)^2/8T^2$ for $k\in\mathbb Z$,
so all singularities of \(L_T\) from $\cosh(\sqrt{2u}\,T)$ lie on the negative real axis. In particular,
\(L_T\) is analytic on the open right half-plane
$
\mathbb H_+
:=
\{u\in\mathbb C:\Re(u)>0\}$.
Applying \cite[Theorem 24.4]{doetsch1974laplace} to the distribution function of $Z_T$, we obtain, at every continuity point \(y>0\) and for any $c>0$, the Bromwich inversion formula:
\[
\P_x(Z_T\leq y)
=
\frac1{2\pi i}
\int_{c-i\infty}^{c+i\infty}
\frac{L_T(u)}{u}e^{uy}\,du.
\]
Let $z:=\frac{y}{T}$,
$
\overline c:=\frac1{8\theta^2}
\text{ and }
c_T^\ast:=\frac1{8z^2}$. The idea is to use $\overline c$ in the Bromwich inversion formula and then use $c_T^\ast$ to identify the correct asymptotic behaviour of that formula. The reason behind our choice of $\overline{c}$ and $c_T^\ast$ will become clearer later on. Note that $c_T^\ast\to\overline c$ since $z\to\theta$ for $T\to\infty$ under the assumptions of Theorem~\ref{thm:uniformasymptotics}.
\subsection{Splitting the integral}
\label{subsec:splitting}
We shall now split the integral into two terms -- the main integral and the error integral. Let us denote $g\equiv g(u):=\sqrt{2u}$. Then, using the notation
\begin{align*}
\psi(u):=uz-\frac{\mu^2}{2}-\frac{g}{2},\quad
d_{x,\mu}(u):=
\frac{1}{u}
\exp\!\left(
-\mu x
-\frac{2ux^2-\mu^2}{2g}
\right), \quad
\rho_T(u):=
\frac{1}{\sqrt{1+e^{-2gT}}}-1,
\end{align*}
as well as
\begin{align*}
\kappa_{T,\mu}(u)
:=
\exp\!\left(B_{T,\mu}(u)\right)-1,\quad
B_{T,\mu}(u)
:=
\frac{e^{-gT}}{1+e^{-2gT}}
\left(
\frac{2ux^2-\mu^2}{g e^{gT}}
+2\mu x
\right)
\end{align*}
we obtain
\begin{align*}
F_T(y )
&:=
\frac{1}{2\pi i}
\int_{c-i\infty}^{c+i\infty}
\frac{e^{-\mu x}}{u\sqrt{\cosh(gT)}}
\,e^{
uz T
-\frac{\mu^2T}{2}
+\frac{\mu x}{\cosh(gT)}\,
-
\frac{(2ux^2-\mu^2)}{2g}\tanh(gT)
}\,du
\\
&=
\frac{1}{2\pi i }
\int_{c-i\infty}^{c+i\infty}
\frac{1}{u\sqrt{\cosh(gT)}}\,
e^{-\mu x +
uz T-\frac{\mu^2T}{2}
+
\frac{2\mu x}{e^{gT}(1+e^{-2gT})}
-
\left(\frac{2ux^2-\mu^2}{2g}\right)
\left(
\frac{1-e^{-2gT}}{1+e^{-2gT}}
\right)
}\,du
\\
&=
\frac{1}{2\pi i}
\int_{c-i\infty}^{c+i\infty}
d_{x,\mu}(u)
\,
\frac{\sqrt{2}e^{T\psi(u)}}{\sqrt{1+e^{-2gT}}}\,
e^{
\left(
\frac{2ux^2-\mu^2}{2g}
\right)
\left(
\frac{2e^{-2gT}}{1+e^{-2gT}}
\right)
+
\frac{2\mu x}{e^{gT}(1+e^{-2gT})}
}\,du
\\
&=
\frac{1}{\sqrt{2}\pi i}
\int_{c-i\infty}^{c+i\infty}
d_{x,\mu}(u)
e^{T\psi(u)}
\left(1+\rho_T(u)\right)(1+\kappa_{T,\mu}(u))
\,du.
\end{align*}

Using this decomposition, we can split the full integral into the {main} and {error} integrals, for $H>0$,
\begin{align*}
    I_T(H)&:=\frac{1}{\sqrt{2}\pi i}\int_{c-iH}^{c+iH}d_{x,\mu}(u)\;e^{T\psi(u)}\,du,\\
    E_T(H)&:=\frac{1}{\sqrt{2}\pi i}\int_{c-iH}^{c+iH} d_{x,\mu}(u)\;e^{T\psi(u)}\left(\rho_T(u)+\kappa_{T,\mu}(u)+\rho_T(u)\kappa_{T,\mu}(u)\right)\,du.
\end{align*}
Then,
\begin{equation*}
    \P_x(Z_T\leq y)=\lim_{H\to\infty}\{I_T(H)+E_T(H)\},
\end{equation*}
where we can choose any $c>0$ for the path integrals, which we will assume implicitly from here on out. We shall show the asymptotic to be $\frac{1}{\sqrt{\pi}}\frac{d_{x,\mu}(\overline{c})}{\sqrt{\psi''(\overline{c})}}T^{-1/2}e^{T\psi(c^\ast_T)}$, which we will later show to be equivalent to the expression claimed in Theorem~\ref{thm:uniformasymptotics}. For this purpose, we consider the following decomposition:
\begin{equation}\label{eq:FTsplitDis}
\left|
F_T( y)
-
\frac{1}{\sqrt{\pi}}\frac{d_{x,\mu}(\overline{c})}{\sqrt{\psi''(\overline{c})}}T^{-1/2}e^{T\psi(c^\ast_T)}
\right|
\le
D^{(1)}+\cdots+D^{(7)},
\end{equation}
where
\begin{align}
    D^{(1)}&:=\limsup\limits_{H\to\infty}\left|E_T(H)\right|,\,\nonumber\\
    D^{(2)}&:=\limsup\limits_{H\to\infty}\left|I_T(H)-I_T(\overline{r})\right|,\nonumber\\
    D^{(3)}&:=\left|I_T(\overline{r})-I_T(T^{-\alpha})\right|,\nonumber\\
    D^{(4)}&:=\left|I_T(T^{-\alpha})-\frac{1}{\sqrt{2}\pi}d_{x,\mu}(\overline{c})\int_{-T^{-\alpha}}^{T^{-\alpha}}e^{T\psi(\overline{c}+iv)}\,dv\right|,\nonumber\\
    D^{(5)}&:=\left|\frac{1}{\sqrt{2}\pi}d_{x,\mu}(\overline{c})\int_{-T^{-\alpha}}^{T^{-\alpha}}e^{T\psi(\overline{c}+iv)}\,dv-\frac{1}{\sqrt{2}\pi}d_{x,\mu}(\overline{c})\int_{-T^{-\alpha}}^{T^{-\alpha}}e^{T\left(\psi(\overline{c})+i\psi'(\overline{c})v-\frac{\psi''(\overline{c})}{2}v^2\right)}\,dv\right|,\nonumber\\
    D^{(6)}&:=\left|\frac{1}{\sqrt{2}\pi}d_{x,\mu}(\overline{c})\int_{-T^{-\alpha}}^{T^{-\alpha}}e^{T\left(\psi(\overline{c})+i\psi'(\overline{c})v-\frac{\psi''(\overline{c})}{2}v^2\right)}\,dv-\frac{1}{\sqrt{2}\pi}d_{x,\mu}(\overline{c})\int_{-T^{-\alpha}}^{T^{-\alpha}}e^{T\left(\psi(c^\ast_T)-\frac{\psi''(\overline{c})}{2}v^2\right)}\,dv\right|,\nonumber\\
    D^{(7)}&:=\left|\frac{1}{\sqrt{2}\pi}d_{x,\mu}(\overline{c})\int_{-T^{-\alpha}}^{T^{-\alpha}}e^{T\left(\psi(c^\ast_T)-\frac{\psi''(\overline{c})}{2}v^2\right)}\,dv-\frac{1}{\sqrt{\pi}}\frac{d_{x,\mu}(\overline{c})}{\sqrt{\psi''(\overline{c})}}T^{-1/2}e^{T\psi(c^\ast_T)}\right|,\nonumber
\end{align}
and $\overline{r}$ shall be defined appropriately later.

\subsection{Control over the error integral (\texorpdfstring{$D^{(1)}$}{D1})}
\label{subsec:errorintegral}

\begin{lem}\label{lem:D1}
Assume \(x^2+\mu^2\leq M T^\beta\) for some $0<\beta<\frac12$, $M>0$ and all sufficiently large $T$. Then, for all \(T\geq T^{(1)}\), where $T^{(1)}$ depends on $c$, $\beta$ and $M$, and $C_E>0$ independent of $T$, we have
\[
\limsup_{H\to\infty}|E_T(H)|
\leq
C_E T^{-3/2} \cdot
d_{x,\mu}(c) T^{-1/2}\,
e^{T\psi(c)}=:f^{(1)}(T)\cdot
d_{x,\mu}(c)\,
T^{-1/2}
e^{T\psi(c)}.
\]
\end{lem}

\begin{proof} We first bound $\rho_T$ and $\kappa_{T,\mu}$ and then show the estimate for the error integral. \\
\textbf{Bound for $d_{x,\mu}$}.
First, we note for later that $d_{x,\mu}(u)$ can be uniformly bounded on any Bromwich line $u=c+iv, c>0$:
\begin{align}\label{eq:dxmubound}
    \abs d_{x,\mu}(u)\abs&=\frac{1}{\abs u\abs}\exp\!\left(-\mu x -\frac{x^2}{2}\Re g(u) +\frac{\mu^2}{2}\Re\frac{1}{g(u)} \right)
    \leq \frac{1}{c}\exp\!\left(-\mu x -\frac{x^2}{2}\sqrt{2c} +\frac{\mu^2}{2\sqrt{2c}} \right) = d_{x,\mu}(c),
\end{align}
where we used  $\Re{g(u)}\geq \sqrt{2c}$ and $\Re\frac{1}{g(u)}\leq \frac{1}{\sqrt{2c}}$ (cf.\ Lemma \ref{lem:generalbounds}(\ref{itm:a}) and (\ref{itm:b})).\\
\textbf{Bound for $\rho_T$}. 
Define $\epsi_T(g(u)):=e^{-2g(u)T}$ so that $\rho_T(u)=\frac{1}{\sqrt{1+\epsi_T (g(u))}}-1$.  Then, by Lemma~\ref{lem:generalbounds}(\ref{itm:a}), 
    \[\abs\epsi_T(g(u))\abs=e^{-2 T\Re g(u)}
        \leq e^{-2 T \sqrt{2c}}\leq \frac{1}{4}\text{ for all }T\geq \frac{\ln 2}{\sqrt{2c}}.\]
        Set $C_\rho:=4\sqrt{2}>0$. Then, Lemma~19 in \cite{ouconditionedrestricted} implies that, for all $T \geq  \frac{\ln 2}{\sqrt{2c}} $,
    \begin{equation} \label{eq:rho_tbound}
        \abs \rho_T(u)\abs=\abs (1+\epsi_T(g(u)))^{-1/2}-1\abs\leq C_\rho\,\abs\epsi_T(g(u))\abs\leq C_\rho\,e^{-2 T\Re g(u)}.
    \end{equation}
\textbf{Bound for $B_{T,\mu}$}.
In order to upper bound $|\kappa_{T,\mu}(u)|=\left|  e^{B_{T,\mu}(u)} - 1 \right|$, we first upper bound $|B_{T,\mu}(u)|$ and then invoke Lemma~18 in \cite{ouconditionedrestricted} after checking the required assumptions.
Using Lemma \ref{lem:generalbounds}(\ref{itm:c}) and (\ref{itm:d}) for $T\geq T_{B_1}= \frac{\ln2}{2\sqrt{2c}}$, 
\begin{align}
        \abs B_{T,\mu}(u)\abs =&\left|\left(\frac{e^{-Tg(u)}}{1+e^{-2Tg(u)}}\right)\,\left(\frac{2ux^2-\mu^2}{g(u)e^{Tg(u)}}+2\mu x\right)\right| \notag
        \\
       \leq~&2e^{-T\Re g(u)}\left(  \left| \frac{g(u)}{e^{Tg(u)}}\right|x^2 +\left|\frac{\mu^2}{g(u)e^{Tg(u)}}\right| +2|\mu x| \right)
       \notag
       \\
       \leq~& 2e^{-T\Re g(u)}\left( \sqrt{2|u|}x^2 {e^{-T \Re g(u)}} +\left|\frac{\mu^2}{\sqrt{2c}}\right|e^{-T\Re g(u)}  +2|\mu x| \right). \label{eqn:uebersicht2}
\end{align}

We deal with each of the terms in (\ref{eqn:uebersicht2}) individually.
For the first term, set $s=\frac{3T}{2}\,\Re\sqrt{2u}$. Then, 
\begin{align*}
&2\sqrt{2|u|}\,x^2e^{-2T\, \Re \sqrt{2u}} \leq2\sqrt{2}\,\Re\sqrt{2u}\,x^2e^{-2T\, \Re \sqrt{2u}}=  2\sqrt{2}\,\frac{2x^2}{3T} \,\left(\frac{3T}{2}\,\Re\sqrt{2u}\, e^{-\frac{3T}{2}\, \Re \sqrt{2u}}\right)e^{-\frac{T}{2}\, \Re \sqrt{2u}} \\
&\leq  \frac{4\sqrt{2}}{3} M\,T^{\beta-1} \,(se^{-s})e^{-\frac{T}{2}\, \Re \sqrt{2u}} \leq \frac{4\sqrt{2} M}{3e} T^{\beta-1} e^{-\frac{T}{2}\, \Re \sqrt{2u}}\leq \frac{4\sqrt{2}M}{3e} e^{-\frac{T}{2}\, \Re \sqrt{2u}},
\end{align*}
using Lemma \ref{lem:generalbounds}(\ref{itm:e}) in the first step, the assumption $x^2\leq M\,T^\beta$, the fact that $se^{-s}\leq \frac{1}{e}$ for all $s\geq0$, $T\geq T_{B_1}$ and $0<\beta<\frac{1}{2}$.
The second term in (\ref{eqn:uebersicht2}) can be estimated similarly since
\[
2\left|\frac{\mu^2}{\sqrt{2c}}\right|e^{-2T\Re g(u)}\leq \left(\frac{2MT^\beta}{\sqrt{2c}}e^{-\frac{3T}{2}\sqrt{2c}}\right)\,e^{-\frac{T}{2}\Re g(u)} \leq e^{-\frac{T}{2}\Re g(u)},
\]
using Lemma~\ref{lem:generalbounds}(\ref{itm:a}), the assumption $\mu^2\leq M T^\beta$ and the fact that $\frac{2M T^\beta}{\sqrt{2c}}e^{-\frac{3T}{2}\sqrt{2c}}\leq 1$ for all $T \geq T_{B_2}$ with $T_{B_2}$ depending on $c$, $\beta$ and $M$.
For the third term in (\ref{eqn:uebersicht2}), the assumption $x^2 +\mu^2\leq MT^{\beta}$ implies $|\mu x|\leq \frac{x^2+\mu^2}{2}\leq\frac{MT^{\beta}}{2}$, hence
\[
4|\mu x| \,e^{-T\Re g(u)} \leq \left(  2MT^\beta \, e^{-\frac{T}{2}\sqrt{2c}}\right) e^{-\frac{T}{2}\Re g(u)} \leq e^{-\frac{T}{2}\Re g(u)}  
\]
since ${2MT^\beta}\,e^{-\frac{T}{2}\sqrt{2c}}\leq 1$ for all $T \geq T_{B_3}$ with $T_{B_3}$ depending on $c$, $\beta$ and $M$.
Therefore, setting $C_B:=\left(   \frac{4\sqrt{2}M}{3e}+2\right)>0$, for $T\geq T_B$ with $T_B:=\max \,\{T_{B_1},T_{B_2},T_{B_3}\}$ depending on $c$, $\beta$ and $M$, we have 
\begin{equation}\label{eq:b_tmubound}
    \abs B_{T,\mu}(u)\abs \leq  C_B \, e^{-\frac{T}{2}\Re g(u)}.
\end{equation}
\textbf{Bound for $\kappa_{T,\mu}$}.
Let $T_E>0$ be such that $e^{-\frac{T\sqrt{2c}}{2}}<\frac{1}{4C_B}$ for all $T\geq T_E$ and $T_E\geq T_B$. For $T\geq T_E$, we therefore have $\abs B_{T,\mu}(u)\abs<C_B\,e^{-\frac{T\sqrt{2c}}{2}}<\frac14$; and Lemma~18 in \cite{ouconditionedrestricted} along with \eqref{eq:b_tmubound} yields
\begin{equation}\label{eq:kappa_tbound}
        \abs\kappa_{T,\mu}(u)\abs=\left|e^{B_{T,\mu}(u)}-1\right|\leq \frac{e}{1-\frac14}\abs B_{T,\mu}(u)\abs\leq \frac{4}{3}e\,C_B\,e^{-\frac{T}{2}\Re g(u)}:= C_\kappa \,e^{-\frac{T}{2}\Re g(u)},
\end{equation}
where $C_\kappa:= \frac{4}{3}e\,C_B>0$ is a constant independent of $T$.
With the bounds \eqref{eq:dxmubound}, \eqref{eq:rho_tbound} and \eqref{eq:kappa_tbound}, we can finally bound the error integral, for all $T\geq T_E$:
\begin{align*}
        \left|E_T(H)\right|&\leq\frac{1}{\sqrt{2}\pi}\int_{c-iH}^{c+iH}\abs d_{x,\mu}(u)\abs e^{T\Re\psi(u)}\left(\abs\rho_T(u)\abs+\abs\kappa_{T,\mu}(u)\abs+\abs\rho_T(u)\abs\,\abs\kappa_{T,\mu}(u)\abs\right)\,du\\
        &\leq \frac{\,d_{x,\mu}(c)}{\sqrt{2}\pi}\int_{c-iH}^{c+iH} e^{T\left(z c-\frac{\mu^2}{2}-\frac{\Re g(u)}{2}\right)}\left(C_\rho\,e^{-2 T\Re g(u)}+ C_\kappa \,e^{- \frac{T}{2}\Re g(u)}+C_\rho C_\kappa \,e^{-\frac{5}{2} T\Re g(u)}\right)\,du\\
        &\leq \frac{\,d_{x,\mu}(c)}{\sqrt{2}\pi}\, e^{T \psi(c)}\int_{c-iH}^{c+iH}C_{\rho,\kappa}\, e^{- \frac{T}{2}\Re g(u)}\,du,
\end{align*}
where $C_{\rho,\kappa}:=C_\rho+C_\kappa+C_\rho C_\kappa$. Therefore, for all $T\geq T_E$,
\begin{equation*}
        \left|E_T(H)\right| \leq C_{\rho,\kappa} \frac{\,d_{x,\mu}(c)}{\sqrt{2}\pi}\, e^{T \psi(c)}\int_{-H}^{H}\, e^{- \frac{T}{2}\Re g(c+iv)}\,dv.
\end{equation*}
We can evaluate the integral using Lemma \ref{lem:generalbounds}(\ref{itm:e}), the symmetry of the integrand and Lemma~20 in \cite{ouconditionedrestricted}:
\begin{align*}
        \int_{-H}^H e^{-  \frac{T}{2}\Re g(c+iv)}\,dv\leq &\int_{-H}^H e^{-  \frac{T}{2}\sqrt{\abs v\abs}}\,dv
        =\frac{16}{T^2}\left(1-e^{-\frac{T}{2}\sqrt{H}}\left(  \frac{T\sqrt{H}}{2}+1\right)\right)\overset{H\to\infty}{\longrightarrow}\frac{16}{T^2}.
\end{align*}
This shows, for all $T\geq T_E$,
\begin{equation*}
        \limsup_{H\to\infty} \left|E_T(H)\right| \leq C_{\rho,\kappa} \frac{\,d_{x,\mu}(c)}{\sqrt{2}\pi}\, e^{T \psi(c)}\frac{16}{T^2}
        =C_E\frac{\,d_{x,\mu}(c)}{T^2}\, e^{T \psi(c)},
\end{equation*}
where $C_E:= \frac{8\sqrt{2}}{\pi} C_{\rho,\kappa} >0$ is a constant independent of $T$, hence proving the lemma.
\end{proof}

\subsection{Controlling the main integral outside the shrinking ball (\texorpdfstring{$D^{(2)}$}{D2} and \texorpdfstring{$D^{(3)}$}{D3})}
\label{subsec:localisation}

\begin{lem}\label{lem:D2}
Let \(r>0\). Then
\[
\limsup_{H\to\infty}\left|I_T(H)-I_T(r)\right|
\leq
f^{(2)}(T)\,
d_{x,\mu}(c)\,
T^{-1/2}\,
e^{T\psi(c)},
\]
for all \(T\geq T^{(2)}:=1\), where
\[
f^{(2)}(T)
:=
\left(
\frac{2\sqrt{r}}{K_{c,r}}+
\frac{2}{K_{c,r}^2}
\right)\,T^{-1/2}
\quad\text{with}~~K_{c,r}:=
\frac{r}
{2(r+c)\sqrt{1+\frac{2c}{r}}}>0.
\]
\end{lem}

\begin{proof}
By definition,
    \begin{equation*}
        I_T(H)-I_T(r)
        =
        \frac{1}{\sqrt{2}\pi}
        \int_{r<|v|\leq H}
        d_{x,\mu}(c+iv)
        e^{T\psi(c+iv)}
        \,dv.
    \end{equation*}

    Using \eqref{eq:dxmubound}, we obtain
    \begin{align}\label{eq:mainintnotlocal}
        \left|I_T(H)-I_T(r)\right|
        &\leq
        \frac{d_{x,\mu}(c)}{\sqrt{2}\pi}
        \int_{r<|v|\leq H}
        e^{T\Re\psi(c+iv)}
        \,dv
        =
        \frac{\sqrt{2}\,d_{x,\mu}(c)}{\pi}e^{T\left(zc-\frac{\mu^2}{2}\right)}
        \int_r^H
        e^{-\frac{T}{2}\Re g(c+iv)}
        \,dv
        \nonumber\\
        &=
        \frac{\sqrt{2}\,d_{x,\mu}(c)}{\pi}
        e^{T\Re\psi(c+ir)}
        \int_r^H
        e^{-\frac{T}{2}\left(\Re g(c+iv)-\Re g(c+ir)\right)}
        \,dv.
    \end{align}

    Define
    $
        f(v)
        :=
        \Re g(c+iv)
        =
        \sqrt{\sqrt{c^2+v^2}+c}.
    $ Then, the derivative of $f$ with respect to $v$ is
    \begin{equation*}
        f'(v)
        =
        \frac{v}
        {2\sqrt{c^2+v^2}\sqrt{\sqrt{c^2+v^2}+c}}.
    \end{equation*}

    For $v\geq r>0$ and $c>0$, $
        \,\sqrt{c^2+v^2}
        \leq
        v+c
    \text{ and so } 
        \sqrt{\sqrt{c^2+v^2}+c}
        \leq
        \sqrt{v+2c}.
    $ Also, $
        \frac{v}{v+c}
        \geq
        \frac{r}{r+c} \text{ and }
        v+2c = \left(1+\frac{2c}{v}\right)v \leq
        \left(1+\frac{2c}{r}\right)v$. Hence,
    \begin{equation*}
        f'(v)
        \geq
        \frac{v}
        {2(v+c)\sqrt{v+2c}}
        \geq
        \frac{r}{2(r+c)}
        \frac{1}{\sqrt{v+2c}}
        \geq
        K_{c,r}\frac{1}{\sqrt{v}}.
    \end{equation*}
    Using this estimate, we get
    \begin{align}
        \Re g(c+iv)-\Re g(c+ir)
        &=
        f(v)-f(r)
        =
        \int_r^v f'(s)ds
        \geq
        K_{c,r}
        \int_r^v \frac{ds}{\sqrt{s}}
        =
        2K_{c,r}
        \left(\sqrt{v}-\sqrt{r}\right).
        \label{eq:regnotlocalbound}
    \end{align}

    Using \eqref{eq:regnotlocalbound} and Lemma 20 in \cite{ouconditionedrestricted}, the last integral in
    \eqref{eq:mainintnotlocal} can be bounded from above:
    \begin{align*}
        &\int_r^H
        e^{-\frac{T}{2}\left(\Re g(c+iv)-\Re g(c+ir)\right)}
        \,dv
        \leq
        \int_r^H
        e^{-TK_{c,r}\left(\sqrt{v}-\sqrt{r}\right)}
        \,dv
        =
        e^{TK_{c,r}\sqrt{r}}
        \int_r^H
        e^{-TK_{c,r}\sqrt{v}}
        \,dv
        \\
        =~&
        e^{TK_{c,r}\sqrt{r}}
        \frac{2}{T^2K_{c,r}^2}
        \Big(
        e^{-TK_{c,r}\sqrt{r}}
        (TK_{c,r}\sqrt{r}+1)
        -
        e^{-TK_{c,r}\sqrt{H}}
        (TK_{c,r}\sqrt{H}+1)
        \Big)
        \\
        \longrightarrow~&
        \frac{2}{T^2K_{c,r}^2}
        (TK_{c,r}\sqrt{r}+1)
        \leq
        \frac{1}{T}
        \left(
        \frac{2\sqrt{r}}{K_{c,r}}
        +
        \frac{2}{K_{c,r}^2}
        \right),
    \end{align*}
    for $H\to\infty$, where we used $T\geq1$. Combining this estimate with \eqref{eq:mainintnotlocal} and using $\frac{\sqrt{2}}{\pi}\leq 1$, we obtain
    \begin{align*}
        \limsup_{H\to\infty}
        \left|I_T(H)-I_T(r)\right|
        \leq
        \left(
        \frac{2\sqrt{r}}{K_{c,r}}
        +
        \frac{2}{K_{c,r}^2}
        \right)\,d_{x,\mu}(c)\,\frac{1}{T}\,
        e^{T\Re \psi(c+ir)}\leq \left(
        \frac{2\sqrt{r}}{K_{c,r}}
        +
        \frac{2}{K_{c,r}^2}
        \right)\,d_{x,\mu}(c)\,\frac{1}{T}\,
        e^{T\psi(c)},
    \end{align*}
    since $\Re\psi(c+ir)-\psi(c) = \frac12(g(c)-\Re g(c+ir))\leq0$ by Lemma\eqref{lem:generalbounds}(a),
    which proves the lemma.
\end{proof}

To tackle the main integral inside a local segment, we use a Taylor expansion of $\psi$ around $c$. For this purpose, note that
\begin{equation}\label{eq:psiderivatives}
    \psi'(u)
    =
    z-\frac{g'(u)}{2}
    =
    z-\frac{1}{2g(u)}
    \quad
\text{ and } \quad
    \psi''(u)
    =
    \frac{1}{2g(u)^2}g'(u)
    =
    \frac{1}{2g(u)^3}.
\end{equation}
In particular, $\psi''(c)$ is strictly positive and we may freely divide by and take roots of $\psi''(c)$ for any $c>0$, a fact we shall repeatedly use throughout the proof. 
Further, for any $y$ satisfying $|y-\theta T|\leq KT^\beta$ for some $K>0$ and $\beta<1/2$ for all sufficiently large $T$, there exists a $T_0>0$ depending on $K$ and $\beta$ such that $z\leq 2\theta$ for all $T\geq T_0$.
With this $T_0$, we can give a bound on the remainder of the second order Taylor approximation of $\psi$. 

\begin{lem}\label{lem:psitaylorline}
    Assume $|y-\theta T|\leq KT^\beta$ for some $K>0$ and $\beta<1/2$ for all sufficiently large $T$. Then there exists a $T_0$ depending only on $K,\beta$, such that for each $c>0$ and $T\geq T_0$, we have
    \begin{equation*}
        \psi(c+iv)
        =
        \psi(c)
        +
        i\psi'(c)v
        -
        \frac{\psi''(c)}{2}v^2
        +
        R_T(c+iv)
    \end{equation*}
    for all $\abs v\abs<\frac{c}{2}$, where
    \begin{equation*}
        \abs R_T(c+iv)\abs
        \leq
        C_\psi
        \frac{\abs v\abs^3}{c^3}
        \frac{1}{1-\frac{2\abs v\abs}{c}}
        \quad\text{ and }\quad
        C_\psi
        :=
        24\theta c
    +
    4{\sqrt{3c}}.
    \end{equation*}
\end{lem}

\begin{proof}
    Since $g$ is holomorphic on $\C\setminus(-\infty,0]$, $\psi$ and therefore $\phi(u)=\psi(u)+\frac{\mu^2}{2}$ are holomorphic on the same domain.
    Since
    $
        B_{c}(c)
        \subset
        \mathbb C\setminus(-\infty,0],
    $
    $\phi$ is holomorphic on $B_{c}(c)$ and
    Taylor's theorem yields
    \begin{equation}\label{eq:phitaylor}
        \phi(u)
        =
        \phi(c)
        +
        \phi'(c)(u-c)
        +
        \frac{\phi''(c)}{2}(u-c)^2
        +
        R_T(u),
    \end{equation}
    where, by Lemma~17 in \cite{ouconditionedrestricted},
    \begin{align}
        \abs R_T(u)\abs
        &\leq
        \max_{s\in[0,2\pi]}
        \left|
        \phi
        \left(
        c+\frac{c}{2}e^{is}
        \right)
        \right|
        \frac{\abs u-c\abs^3}{(c/2)^3}
        \frac{1}{1-\frac{\abs u-c\abs}{c/2}}.
        \label{eq:phitaylorbound}
    \end{align}
The triangle inequality together with $z\leq 2\theta$ yields
\begin{align*}
    \max_{s\in[0,2\pi]}
\left|
\phi\left(
c+\frac{c}{2}e^{is}
\right)
\right|
    &\leq
    z
    \max_{s\in[0,2\pi]}
    \left|
    c+\frac{c}{2}e^{is}
    \right|
    +
    \frac{1}{2}
    \max_{s\in[0,2\pi]}
    \sqrt{2\left|c+\frac{c}{2}e^{is}\right|}
    \leq
    3\theta c
    +
    \frac{\sqrt{3c}}{2}=: \frac{1}{8} C_\psi, 
\end{align*}
which together with (\ref{eq:phitaylorbound}) shows the claimed bound for $R_T$.
Since $\psi'=\phi'$ and $\psi''=\phi''$, subtracting $\frac{\mu^2}{2}$ on both sides of \eqref{eq:phitaylor} yields the claimed expansion of $\psi$ around $c$.
\end{proof}

From here on out, let $C_\psi$ be given by Lemma~\ref{lem:psitaylorline} and let $r$ depend on $c$ in the following way:
\begin{equation}\label{eq:defr(c)}
    r=r(c)
    :=
    \min\left\{
    \frac{c}{4},
    \frac{\psi''(c)\,c^3}{8C_\psi}
    \right\}.
\end{equation}

\begin{lem}\label{lem:D3}
Assume $|y-\theta T|\leq KT^\beta$ for some $K>0$ and $\beta<1/2$ for all sufficiently large $T$. Let \(\alpha>0\) and $T_0$ be given by Lemma~\ref{lem:psitaylorline}. Then, for all $T\geq T^{(3)}:=\max\{T_0,r^{-1/\alpha}\}$,
\[
\left|I_T(r)-I_T(T^{-\alpha})\right|
\leq
\frac{\sqrt{2}r}{\pi}\,
T^{1/2}
\exp\!\left(
-\frac{\psi''(c)}{4}T^{1-2\alpha}
\right)
\cdot
d_{x,\mu}(c)\,
T^{-1/2}
e^{T\psi(c)}
=:
f^{(3)}(T)\cdot
d_{x,\mu}(c)\,
T^{-1/2}
e^{T\psi(c)}.
\]
\end{lem}

\begin{proof}
    We will first see the reason for the definition of $r$ in \eqref{eq:defr(c)}. Since
    $
        \psi(c),\,
        \psi'(c), \text{ and }
        \psi''(c)
    $
    are real (and hence
    $i\psi'(c)v$ is purely imaginary),
    the Taylor expansion of $\psi$
    (Lemma \ref{lem:psitaylorline}) yields
    \begin{align*} 
        \Re\psi(c+iv)
        =
        \psi(c)
        -
        \frac{\psi''(c)}{2}v^2
        +
        \Re R_T(c+iv).
    \end{align*}

    Moreover,
    \begin{align*}
        \Re R_T(c+iv)
        &\leq
        \left|R_T(c+iv)\right|\leq
        C_\psi
        \frac{|v|^3}{c^3}
        \frac{1}{1-\frac{2|v|}{c}}\leq
        C_\psi
        \frac{r}{c^3}
        \frac{1}{1-\frac{2r}{c}}
        v^2\leq
        2C_\psi
        \frac{r}{c^3}
        v^2\leq
        \frac{\psi''(c)}{4}
        v^2,
    \end{align*}
    on $|v|\leq r$, where we used
    $r\leq c/4$
    and
    $r\leq \frac{\psi''(c)\,c^3}{8C_\psi}$. This yields
    \begin{equation}\label{eq:repsibound}
        \Re\psi(c+iv)
        \leq
        \psi(c)
        -
        \frac{\psi''(c)}{4}v^2
    \end{equation}
    on $|v|\leq r$. We use this and \eqref{eq:dxmubound} to see
    \begin{align*}
        &\left|I_T(r)-I_T(T^{-\alpha})\right|
        \leq
        \frac{1}{\sqrt{2}\pi}
        \int_{T^{-\alpha}<|v|\leq r}
        |d_{x,\mu}(c+iv)|
        e^{T\Re\psi(c+iv)}
        \,dv\\
        \leq~&
        \frac{1}{\sqrt{2}\pi}
        d_{x,\mu}(c)
        \int_{T^{-\alpha}<|v|\leq r}
        e^{T\left(
        \psi(c)
        -
        \frac{\psi''(c)}{4}v^2
        \right)}
        \,dv
        \leq
        \frac{\sqrt{2}r}{\pi}
        d_{x,\mu}(c)
        e^{T\psi(c)}
        e^{-\frac{\psi''(c)}{4}T^{1-2\alpha}},
    \end{align*}
    where we used the symmetry of the integrand and $r-T^{-\alpha}\leq r$. This proves the lemma.
\end{proof}

\subsection{Evaluation near the saddle-point (\texorpdfstring{$D^{(4)}$}{D4}--\texorpdfstring{$D^{(7)}$}{D7})}
\label{subsec:shrinkingball}

\begin{lem}\label{lem:D4}
Assume $|y-\theta T|\leq KT^\beta$ and $\mu^2+x^2\leq MT^\beta$ for some $K,M>0$ and $0<\beta<1/2$ for all sufficiently large $T$.
Let \(\alpha>\beta\) and \(r\) be defined by
\eqref{eq:defr(c)}. Then there is a  $T^{(4)}$ depending only on $\alpha,\beta,c,M,K$ and there is a $C_{c,M}>0$ depending on $c$ and $M$ such that, for all \(T\geq T^{(4)}\),
\begin{align*}
& \left|
I_T(T^{-\alpha})
-
\frac{d_{x,\mu}(c)}{\sqrt{2}\pi}
\int_{-T^{-\alpha}}^{T^{-\alpha}}
e^{T\psi(c+iv)}
\,dv
\right|
\leq
C_{c,M}\,T^{\beta-\frac12}\cdot
d_{x,\mu}(c)\,
T^{-1/2}
e^{T\psi(c)}=:f^{(4)}(T)\,
\cdot d_{x,\mu}(c)\,
T^{-1/2}
e^{T\psi(c)}.
\end{align*}
\end{lem}

\begin{proof}
    Define $E_{x,\mu}(u):=\exp\left(-\mu x+\frac{\mu^2}{2g(u)}-\frac{x^2u}{g(u)}\right)$,
    Then $d_{x,\mu}(u)=E_{x,\mu}(u)/u$.
    Hence,
    \begin{align*}
        &
        I_T(T^{-\alpha})
        -
        \frac{1}{\sqrt{2}\pi}
        d_{x,\mu}(c)
        \int_{-T^{-\alpha}}^{T^{-\alpha}}
        e^{T\psi(c+iv)}
        \,dv
        =
        \frac{1}{\sqrt{2}\pi}
        \int_{-T^{-\alpha}}^{T^{-\alpha}}
        \left(
            \frac{E_{x,\mu}(c+iv)}{c+iv}
            -
            \frac{E_{x,\mu}(c)}{c}
        \right)
        e^{T\psi(c+iv)}
        \,dv.
    \end{align*}

    We split this expression into two parts:
    \begin{align*}
        I_T^{(1)}
        &:=
        \frac{1}{\sqrt{2}\pi}
        \int_{-T^{-\alpha}}^{T^{-\alpha}}
        \left(\frac{1}{c+iv}-\frac{1}{c}\right)
        E_{x,\mu}(c+iv)
        e^{T\psi(c+iv)}
        \,dv,
        \\
        I_T^{(2)}
        &:=
        \frac{1}{\sqrt{2}\pi}\frac{1}{c}
        \int_{-T^{-\alpha}}^{T^{-\alpha}}
        \left(
            E_{x,\mu}(c+iv)
            -
            E_{x,\mu}(c)
        \right)
        e^{T\psi(c+iv)}
        \,dv.
    \end{align*}

    Then
    \[
        I_T(T^{-\alpha})
        -
        \frac{1}{\sqrt{2}\pi}
        d_{x,\mu}(c)
        \int_{-T^{-\alpha}}^{T^{-\alpha}}
        e^{T\psi(c+iv)}
        \,dv
        =
        I_T^{(1)}+I_T^{(2)}.
    \]

    We first estimate $I_T^{(1)}$.
    For this purpose we note that, with \eqref{eq:dxmubound} we have
\begin{align}\label{eq:boundAdiff}
    \left| \frac{1}{c+iv}-\frac{1}{c} \right| \left|E_{x,\mu}(c+iv)\right|=\frac{|v|}{c}\left| \frac{E_{x,\mu}(c+iv)}{c+iv} \right|=\frac{|v|}{c}\left|d_{x,\mu}(c+iv)\right|\leq \frac{|v|}{c}d_{x,\mu}(c).
\end{align}
    For $T\geq T^{(3)}$ given by Lemma~\ref{lem:D3}, by the arguments in the proof of that lemma, we have \eqref{eq:repsibound} which yields
    \begin{align}\label{eq:vev2integral}
        \int_{-T^{-\alpha}}^{T^{-\alpha}}\abs v\abs\,e^{T\Re\psi(c+iv)}\,dv
        &\leq \int_{-T^{-\alpha}}^{T^{-\alpha}}\abs v\abs\,e^{T\left(\psi(c)-\frac{\psi''(c)}{4}v^2\right)}\,dv
        =e^{T\psi(c)}\int_{-T^{-\alpha}}^{T^{-\alpha}}\abs v\abs\,e^{-T\frac{\psi''(c)}{4}v^2}\,dv\nonumber\\
        &=2e^{T\psi(c)}\frac{2}{T\psi''(c)}\left(\left.-e^{-T\frac{\psi''(c)}{4}v^2}\right|_0^{T^{-\alpha}}\right)
        \leq \frac{4}{\psi''(c)}\,T^{-1}\,e^{T\psi(c)}
    \end{align}
since $\psi''(c)>0$. Hence, due to \eqref{eq:boundAdiff} and \eqref{eq:vev2integral}, with the trivial estimate $1\leq T^\beta$, we have
    \begin{equation}\label{eq:i_t1bound}
        |I_T^{(1)}|\leq
        \frac{1}{\sqrt2\pi c}d_{x,\mu}(c)
        \int_{-T^{-\alpha}}^{T^{-\alpha}}
        |v|\,
        e^{T\Re\psi(c+iv)}
        \,dv
        \leq
        \frac{2\sqrt{2}}
        {\pi c\,\psi''(c)}
        d_{x,\mu}(c)
        T^{\beta-1}
        e^{T\psi(c)},
    \end{equation}
    for $T\geq\max\{T^{(3)},1\}$.
    We now estimate $I_T^{(2)}$.
    Define
    \[H(u):=\frac{\mu^2}{2g(u)}-\frac{x^2u}{g(u)}=\frac{\mu^2}{2g(u)}-\frac{x^2}{2}g(u).\]
    Then we have $
        E_{x,\mu}(u)
        =
        e^{-\mu x}e^{H(u)}$. 
    Differentiating $H(u)$ with respect to $u$, we obtain
    \begin{align*}
        H'(u)=-\frac{\mu^2}{2g(u)^2}g'(u)-\frac{x^2}{2}g'(u)
        =-\frac{\mu^2}{2g(u)^3}-\frac{x^2}{2g(u)}.
    \end{align*}
    From Lemma~\ref{lem:generalbounds}~(a) we get $|g(c+iv)|\geq\Re g(c+iv)\geq \sqrt{2c}$ for $v\in\R$ and thus we get the following bound on $H'$ along this vertical line,
    \begin{align*}
        |H'(c+iv)|&\leq\frac{\mu^2}{2}\frac{1}{(2c)^{3/2}}+\frac{x^2}{2}\frac{1}{(2c)^{1/2}}\leq C_c(\mu^2+x^2),
    \end{align*}
    for some constant $C_c>0$ depending only on $c$.
    Since $H$ is holomorphic on $B_r(c)$, the fundamental theorem of calculus along the line segment
    $t\longmapsto c+itv,\,t\in[0,1],$
    yields
    \begin{align*}
        |H(c+iv)-H(c)|
        =\left|\int_0^1\frac{d}{dt}H(c+itv)\,dt\right|
        \leq|v|\int_0^1\left|H'(c+itv)\right|\,dt
        \leq C_c(\mu^2+x^2)|v|
    \end{align*}
    for every $|v|< r$.
    Since
    $
        \mu^2+x^2\leq MT^\beta
     \text{ and } 
        |v|\leq T^{-\alpha}<r,
    $
    we obtain
    $
        |H(c+iv)-H(c)|
        \leq
        C_c MT^{\beta-\alpha}.
    $
    Since $\beta<\alpha$, the right-hand side converges to $0$.
    Hence, for $T\geq T^{(4)}$, where $T^{(4)}\geq \max\{T^{(3)},1\}$ depends on $\alpha, \beta$, $c,M,K$,
    $
        |H(c+iv)-H(c)|\leq\frac12.
    $ Applying Lemma~18 in \cite{ouconditionedrestricted},
    \[
        \left|
            e^{H(c+iv)-H(c)}-1
        \right|
        \leq
        2e\,|H(c+iv)-H(c)|
        \leq
        2e\,C_c(\mu^2+x^2)|v|.
    \]
    Therefore, together with \eqref{eq:vev2integral} and $\mu^2+x^2\leq MT^\beta$ we have
    \begin{align}\label{eq:i_t2bound}
        |I_T^{(2)}|
        &\leq\frac{1}{\sqrt2\pi}\frac{1}{c}\,e^{-\mu x}e^{H(c)}\int_{-T^{-\alpha}}^{T^{-\alpha}}\left|e^{H(c+iv)-H(c)}-1\right|e^{T\Re\psi(c+iv)}
        \,dv \nonumber
        \\
        &\leq \frac{2eC_c}{\sqrt2\pi}\frac{1}{c}(\mu^2+x^2)\,E_{x,\mu}(c)\int_{-T^{-\alpha}}^{T^{-\alpha}}|v|e^{T\Re\psi(c+iv)}\,dv
        \leq
        \frac{4\sqrt2 e\,C_c M}
        {\pi\psi''(c)}
        d_{x,\mu}(c)
        T^{\beta-1}
        e^{T\psi(c)}.
    \end{align}
    Combining the estimates \eqref{eq:i_t1bound} and \eqref{eq:i_t2bound} completes the proof.
\end{proof}

\begin{lem}\label{lem:D5}
Assume $|y-\theta T|\leq KT^\beta$ for some $K>0$ and $0<\beta<1/2$ for all sufficiently large $T$.
Let \(\alpha>\frac13\), \(r\) be defined by
\eqref{eq:defr(c)} and $T_0$ be given by Lemma~\ref{lem:psitaylorline}.
Then, with $C'_{c,\theta}>0$ depending only on $c$ and $\theta$, for all
$
T\geq T^{(5)}
:=
\max\left\{
\left(\frac{2}{r}\right)^{1/\alpha},
\left(
\frac{8C_\psi}{3c^3}
\right)^{\frac1{3\alpha-1}},
T_0
\right\}
$,
\begin{multline*}
    \left|
\frac{1}{\sqrt2\pi}
d_{x,\mu}(c)
\int_{-T^{-\alpha}}^{T^{-\alpha}}
e^{T\psi(c+iv)}
\,dv
-
\frac{1}{\sqrt2\pi}
d_{x,\mu}(c)
\int_{-T^{-\alpha}}^{T^{-\alpha}}
e^{T\left(
\psi(c)
+i\psi'(c)v
-\frac{\psi''(c)}{2}v^2
\right)}
\,dv
\right|\\
\leq
C'_{c,\theta}\,T^{1/2-3\alpha}\,
d_{x,\mu}(c)\,
e^{T\psi(c)}=:
f^{(5)}(T)\,
d_{x,\mu}(c)\,
T^{-1/2}
e^{T\psi(c)}.
\end{multline*}
\end{lem}

\begin{proof}
    We use the Taylor expansion of $\psi$
    from Lemma \ref{lem:psitaylorline}. By definition of $r$ and the assumption on $T$,
    $|v|\leq T^{-\alpha}\leq\frac{r}{2}\leq\frac{c}{8}$ which implies $1-\frac{2|v|}{c}\geq\frac34$.
    Using Lemma \ref{lem:psitaylorline}, we then get
    \begin{align*}
        |TR_T(c+iv)|
        &\leq
        T
        C_\psi
        \frac{|v|^3}{c^3}
        \frac{1}{1-\frac{2|v|}{c}}\leq
        \frac{4C_\psi}{3c^3}
        T|v|^3 \leq
        \frac{4C_\psi}{3c^3}
        \,T^{1-3\alpha}
        \leq
        \frac12,
    \end{align*}
    where we used $T\geq \left(\frac{8C_\psi}{3c^3}\right)^{\frac1{3\alpha-1}}$,
    together with $\alpha>\frac13$ implying
    $1-3\alpha<0$. Applying Lemma~18 in \cite{ouconditionedrestricted},
    \[
        \left|
            e^{TR_T(c+iv)}-1
        \right|
        \leq
        2e\,|TR_T(c+iv)|
        \leq
        \frac{8eC_\psi}{3c^3}\,T^{1-3\alpha}.
    \]

    Using this estimate and the Taylor expansion,
    \begin{align*}
        &
        \left|
        \frac{1}{\sqrt2\pi}
        d_{x,\mu}(c)
        \int_{-T^{-\alpha}}^{T^{-\alpha}}
        \left(
        e^{T\psi(c+iv)}
        -
        e^{T\left(
            \psi(c)
            +
            i\psi'(c)v
            -
            \frac{\psi''(c)}{2}v^2
        \right)}
        \right)
        \,dv
        \right|
        \\
        \leq~&
        \frac{d_{x,\mu}(c)}{\sqrt2\pi}
        \int_{-T^{-\alpha}}^{T^{-\alpha}}
        e^{T\left(
            \psi(c)
            -
            \frac{\psi''(c)}{2}v^2
        \right)}
        \left|
            e^{TR_T(c+iv)}-1
        \right|
        \,dv
        \\
        \leq~&
        \frac{d_{x,\mu}(c)}{\sqrt2\pi}
        \frac{8eC_\psi}{3c^3}
        T^{1-3\alpha}
        e^{T\psi(c)}
        \int_{-T^{-\alpha}}^{T^{-\alpha}}
        e^{-T\frac{\psi''(c)}{2}v^2}
        \,dv
        \\
        =~&
        \frac{d_{x,\mu}(c)}{\sqrt2\pi}\frac{8eC_\psi}{3c^3\sqrt{\psi''(c)}}
        T^{\frac12-3\alpha}
        e^{T\psi(c)}
        \int_{-\sqrt{\psi''(c)}T^{1/2-\alpha}}
             ^{\sqrt{\psi''(c)}T^{1/2-\alpha}}
        e^{-s^2/2}
        \,ds
        \\
        \leq~&
        C'_{c,\theta}
        \,
        d_{x,\mu}(c)
        \,
        T^{\frac12-3\alpha}
        e^{T\psi(c)}.\qedhere
    \end{align*}
\end{proof}

Now fix an $\alpha>\frac13$ with $\beta<\alpha<1/2$.
Note that, if
$
|y-\theta T|\le KT^\beta,
$
we also have
$
|z-\theta|
\le KT^{\beta-1}
$
for all sufficiently large $T$. 
Therefore $\beta<\frac12$ implies the existence of some $T^{(0)}$ such that for all $T\ge T^{(0)}$ we have
\begin{equation}\label{eq:zthetabound}
\frac{\theta}{2}
<
\theta-KT^{\beta-1}
\le z
\le
\theta+KT^{\beta-1}
<
\frac{3\theta}{2}.
\end{equation}
Note that since the definition of $T^{(0)}$ only depends on $K$ and $\beta$, we can choose the $T^{(i)}$ in the previous lemmas such that $T^{(i)}\geq T^{(0)}$ for $i=1,\ldots,5$.
Further, as we defined $\overline c:=\frac1{8\theta^2}$ and $c_T^\ast:=\frac1{8z^2}$, we have
$g(\overline c)=\sqrt{2\overline c}=\frac1{2\theta}$ and
$g(c_T^\ast)=\sqrt{2c_T^\ast}=\frac1{2z}$.
With these observations and \eqref{eq:zthetabound} it follows that
\begin{align}
\left|
\psi(\overline c)
-
\psi(c_T^\ast)
\right|
&=
\left|
z\overline c-\frac{\mu^2}{2}-\frac{g(\overline c)}2
-
\left(
zc_T^\ast-\frac{\mu^2}{2}-\frac{g(c_T^\ast)}2
\right)
\right|
=
\left|
\frac z{8\theta^2}
-\frac1{4\theta}
-\frac1{8z}
+\frac1{4z}
\right|
\nonumber\\
&=
\frac1{8z\theta^2}
\left|
z^2-2z\theta-\theta^2+2\theta^2
\right|
=
\frac{|z-\theta|^2}{8z\theta^2 }\le
\frac{K^2}{4\theta^3}
T^{2\beta-2},
\label{eq:psidiffbound}
\end{align}
for all $T\ge T^{(0)}$, having used the lower bound in (\ref{eq:zthetabound}). Moreover, by \eqref{eq:psiderivatives},
\begin{equation}\label{eq:psiderivativebound}
\left|
\psi'(\overline c)
\right|
=
\left|
z-\frac1{2g(\overline c)}
\right|
=
|z-\theta|
\le
KT^{\beta-1}.
\end{equation}

For the subsequent arguments we keep in mind that
$
\psi''(\overline c)
=
\frac1{2g(\overline c)^3}
=
4\theta^3
>
0.
$ Set, as in (\ref{eq:defr(c)}),
\begin{align}\label{def:roverline}
\overline r:=r(\overline c)
=
\min\left\{
\frac{\overline c}{4},
\frac{\psi''(\overline c)\,\overline c^3}{8C_\psi}
\right\}.
\end{align}

\begin{lem}\label{lem:D6}
Assume $|y-\theta T|\leq KT^\beta$ for some $K>0$ and $0<\beta<1/2$ for all sufficiently large $T$.
Let \(\beta<\alpha<1/2\) and $T^{(0)}$ be given by \eqref{eq:zthetabound}. Then there exists \(T^{(6)}\geq T^{(0)}\) such that, for all $T\geq T^{(6)}$,
\begin{multline*}
    \left|\frac{1}{\sqrt{2}\pi}d_{x,\mu}(\overline{c})\int_{-T^{-\alpha}}^{T^{-\alpha}}e^{T\left(\psi(\overline{c})+i\psi'(\overline{c})v-\frac{\psi''(\overline{c})}{2}v^2\right)}\,dv-\frac{1}{\sqrt{2}\pi}d_{x,\mu}(\overline{c})\int_{-T^{-\alpha}}^{T^{-\alpha}}e^{T\left(\psi(c^\ast_T)-\frac{\psi''(\overline{c})}{2}v^2\right)}\,dv\right|\\
    \leq
\frac{2e}{\sqrt{\pi\,\psi''(\overline c)}}
\left(
\frac{K^2}{4\theta^3}T^{2\beta-1}
+
KT^{\beta-\alpha}
\right)\,
d_{x,\mu}(\overline c)\,
T^{-1/2}
e^{T\psi(c_T^\ast)}
=: f^{(6)}(T)\,
d_{x,\mu}(\overline c)\,
T^{-1/2}
e^{T\psi(c_T^\ast)}.
\end{multline*}
\end{lem}

\begin{proof}
Combining \eqref{eq:psidiffbound} and \eqref{eq:psiderivativebound} yields
\begin{equation*}
\left|
T\Bigl(
\psi(\overline c)
-
\psi(c_T^\ast)
+
i\psi'(\overline c)v
\Bigr)
\right|
\leq
T\left(
\frac{K^2}{4\theta^3}T^{2\beta-2}
+
KT^{\beta-1}|v|
\right)
\leq
\frac{K^2}{4\theta^3}T^{2\beta-1}
+
KT^{\beta-\alpha},
\end{equation*}
for all \(T\geq T^{(0)}\) and \(|v|\leq T^{-\alpha}\). Since \(\beta<\alpha<\frac12\), there exists \(T^{(6)}\geq T^{(0)}\) such that
the expression in the last display is bounded by $\frac{1}{2}$
for all \(T\geq T^{(6)}\) and \(|v|\leq T^{-\alpha}\).
Hence Lemma~18 in \cite{ouconditionedrestricted} implies that for such $T$ and $v$
\begin{align}
\left|
e^{T(\psi(\overline c)-\psi(c_T^\ast)+i\psi'(\overline c)v)}
-1
\right|&\leq
2e
\left|
T\Bigl(
\psi(\overline c)
-
\psi(c_T^\ast)
+
i\psi'(\overline c)v
\Bigr)
\right|
\leq
2e
\left(
\frac{K^2}{4\theta^3}T^{2\beta-1}
+
KT^{\beta-\alpha}
\right).
\label{eq:D6exppsidiff}
\end{align}
Using \eqref{eq:D6exppsidiff}, we obtain
\begin{align*}
&
\left|
\frac{1}{\sqrt2\pi}
d_{x,\mu}(\overline c)
\int_{-T^{-\alpha}}^{T^{-\alpha}}
e^{T\left(
\psi(\overline c)
+i\psi'(\overline c)v
-\frac{\psi''(\overline c)}2v^2
\right)}
\,dv
-
\frac{1}{\sqrt2\pi}
d_{x,\mu}(\overline c)
\int_{-T^{-\alpha}}^{T^{-\alpha}}
e^{T\left(
\psi(c_T^\ast)
-\frac{\psi''(\overline c)}2v^2
\right)}
\,dv
\right|
\\
\leq~&
\frac{d_{x,\mu}(\overline c)}{\sqrt2\pi}
e^{T\psi(c_T^\ast)}
\int_{-T^{-\alpha}}^{T^{-\alpha}}
\left|
e^{T(\psi(\overline c)-\psi(c_T^\ast)+i\psi'(\overline c)v)}
-1
\right|
e^{-T\frac{\psi''(\overline c)}2v^2}
\,dv
\\
\leq~&
\frac{e\sqrt2}{\pi}
d_{x,\mu}(\overline c)
\left(
\frac{K^2}{4\theta^3}T^{2\beta-1}
+
KT^{\beta-\alpha}
\right)
e^{T\psi(c_T^\ast)}
\int_{-T^{-\alpha}}^{T^{-\alpha}}
e^{-T\frac{\psi''(\overline c)}2v^2}
\,dv\\
\leq~&
\frac{2e}{\sqrt{\pi\,\psi''(\overline c)}}
\left(
\frac{K^2}{4\theta^3}T^{2\beta-1}
+
KT^{\beta-\alpha}
\right)
d_{x,\mu}(\overline c)
T^{-1/2}
e^{T\psi(c_T^\ast)},
\end{align*}
where we estimated the integral over $\R$ of a normal density.
%
\end{proof}

\begin{lem}\label{lem:D7}
Assume $|y-\theta T|\leq KT^\beta$ for some $K>0$ and $0<\beta<1/2$ for all sufficiently large $T$.
Let \(0<\alpha<\frac12\) and $T^{(0)}$ be given by \eqref{eq:zthetabound}. Then there exists \(T^{(7)}\geq T^{(0)}\) such that, for all $T\geq T^{(7)}$,
\begin{align*}
    &\left|\frac{1}{\sqrt{2}\pi}d_{x,\mu}(\overline{c})\int_{-T^{-\alpha}}^{T^{-\alpha}}e^{T\left(\psi(c^\ast_T)-\frac{\psi''(\overline{c})}{2}v^2\right)}\,dv-\frac{1}{\sqrt{\pi}}\,\frac{d_{x,\mu}(\overline{c})}{\sqrt{\psi''(\overline{c})}}\,T^{-1/2}\,e^{T\psi(c^\ast_T)}\right|\nonumber\\
&\leq
 \frac{\sqrt{2}}{\pi\,\psi''(\overline{c})}\,T^{\alpha-1/2}\cdot\,d_{x,\mu}(\overline{c})\,T^{-1/2}\,e^{T\psi(c^\ast_T)}
=: f^{(7)}(T)\cdot \,d_{x,\mu}(\overline{c})\,T^{-1/2}\,e^{T\psi(c^\ast_T)}.\end{align*}
\end{lem}

\begin{proof}
Using the substitution
$
s=\sqrt{T\psi''(\overline c)}\,v,
$
we obtain
\begin{align*}
\left|
\frac{1}{\sqrt{2\pi}}
\sqrt{T\psi''(\overline c)}
\int_{-T^{-\alpha}}^{T^{-\alpha}}
e^{-T\frac{\psi''(\overline c)}2v^2}\,dv
-1\right| 
&=
\frac{1}{\sqrt{2\pi}}
\int_{|s|\geq
\sqrt{\psi''(\overline c)}\,T^{1/2-\alpha}}
e^{-s^2/2}\,ds.
\end{align*}

Let \(Y\sim N(0,1)\). Then the last quantity equals
$
\,\mathbb P\!\left(
|Y|\geq
\sqrt{\psi''(\overline c)}\,T^{1/2-\alpha}
\right).
$ Using Markov's inequality,
\[
\mathbb P(|Y|\geq \sqrt{\psi''(\overline c)}\,T^{1/2-\alpha})
\leq
\frac{\,\mathbb E|Y|}{\sqrt{\psi''(\overline c)}\,T^{1/2-\alpha}}
=
\frac{\sqrt2}
{\sqrt{\pi}\sqrt{\psi''(\overline c)}}
T^{\alpha-\frac12}.
\]

Hence, for all \(T\geq T^{(7)}:=T^{(0)}\),
\begin{align*}
&\left|
\frac{1}{\sqrt2\pi}
d_{x,\mu}(\overline c)
\int_{-T^{-\alpha}}^{T^{-\alpha}}
e^{T\left(
\psi(c_T^\ast)
-\frac{\psi''(\overline c)}2v^2
\right)}
\,dv
-
\frac{1}{\sqrt{\pi}}
\frac{d_{x,\mu}(\overline c)}
{\sqrt{\psi''(\overline c)}}
T^{-1/2}
e^{T\psi(c_T^\ast)}
\right|
\\
=~&\frac{1}{\sqrt{\pi}}
\frac{d_{x,\mu}(\overline c)}
{\sqrt{\psi''(\overline c)}}
T^{-1/2}
e^{T\psi(c_T^\ast)}
\left|
\frac{1}{\sqrt{2\pi}}
\sqrt{T\psi''(\overline c)}
\int_{-T^{-\alpha}}^{T^{-\alpha}}
e^{-T\frac{\psi''(\overline c)}2v^2}\,dv
-1
\right|
\\
\leq~&
\frac{\sqrt2}
{\pi\psi''(\overline c)}
T^{\alpha-\frac12}
d_{x,\mu}(\overline c)
T^{-1/2}
e^{T\psi(c_T^\ast)}. \qedhere
\end{align*}
\end{proof}

\subsection{Final asymptotics}
\label{subsec:finalasymptotic}

We can now combine the previous lemmas on the way to proving Theorem~\ref{thm:uniformasymptotics}.
\begin{lem}\label{lem:Difirelation}
    Fix $\theta>0$, $0<\beta<\frac{1}{2}$ and $M,K>0$ and let $T^{(0)}$ be given by \eqref{eq:zthetabound}. Then, for each $i=1,\ldots,7$ there exists a function $g^{(i)}(T)\geq 0$ with $g^{(i)}(T)\to 0$, for $T\to\infty$, and there exists a $T^{(i)}>T^{(0)}$ such that, for all $T\geq T^{(i)}$, we have for any $x=x(T),\mu=\mu(T)$ and $y=y(T)>0$ with
    $x^2+\mu^2\leq MT^{\beta}$ and $\abs y-\theta T\abs\leq KT^\beta$ for all $T\geq T^{(i)}$, the following estimates:
    \begin{equation*}
        D^{(i)}\leq g^{(i)}(T)\,d_{x,\mu}(\overline{c})\,T^{-1/2}\,e^{T\psi(c^\ast_T)},\qquad i=1,\ldots,7.
    \end{equation*}
\end{lem}

\begin{proof} Fix $\alpha>\frac{1}{3}$ such that $\beta <\alpha<\frac{1}{2}$. The bounds for the quantities \(D^{(i)}\), \(i=1,\dots,7\), were established in
Lemmas~\ref{lem:D1}, \ref{lem:D2} and Lemmas~\ref{lem:D3}--\ref{lem:D7}. More precisely, for $i=1,\ldots,5$ there exist
\(T^{(i)}\geq T^{(0)}\) such that
\[
D^{(i)}
\leq
f^{(i)}(T)\,
d_{x,\mu}(\overline c)\,
T^{-1/2}e^{T\psi(\overline{c})}
=f^{(i)}(T)\,e^{T\left(\psi(\overline{c})-\psi(c^\ast_T)\right)}
d_{x,\mu}(\overline c)\,
T^{-1/2}
e^{T\psi(c_T^\ast)},
\]
for all \(T\geq T^{(i)}\), i.e. $g^{(i)}(T):=f^{(i)}(T)\,e^{T\left(\psi(\overline{c})-\psi(c^\ast_T)\right)}$ where the non-negative functions
\(f^{(i)}\) are given explicitly in the respective lemmas for $i=1,\ldots,5$.
For $i=6,7$ the corresponding lemmas give the existence of $T^{(i)}\geq T^{(0)}$ such that the wanted estimates hold with $g^{(i)}:=f^{(i)}$.
It therefore remains to verify that each
\(g^{(i)}(T)\) converges to zero as
\(T\to\infty\). To this end, we shall show that $f_i(T)\to 0$ for all $i=1,\ldots,7$, while $e^{T\left(\psi(\overline{c})-\psi(c^\ast_T)\right)} \to 1$. To see the latter, note that \eqref{eq:psidiffbound} gives
    $\bigl|T(\psi(\overline c)-\psi(c_T^\ast))\bigr|
\leq
\frac{K^2}{4\theta^3}T^{2\beta-1}$,
which tends to zero because $\beta<1/2$.

We are left with the \(f^{(i)}\). For \(f^{(1)}\), using Lemma~\ref{lem:D1} 
gives \(f^{(1)}(T)\to 0\).
Similarly, using the definition of $\overline r$ given in \eqref{def:roverline}, Lemma~\ref{lem:D2} yields $f^{(2)}(T)\to 0.$
Using the definition of $\overline{r}$ given in \eqref{def:roverline}, Lemma~\ref{lem:D3} along with \(\alpha<\frac12\) and $\psi''(\overline{c})>0$ gives \(f^{(3)}(T)\to 0\).
Further, Lemma~\ref{lem:D4} along with \(\beta<\frac12\) yields $f^{(4)}(T)\to 0$.
Similarly, Lemma~\ref{lem:D5} along with $ \alpha>\frac{1}{3}$ gives $f^{(5)}(T)\to 0$.
Since \(\beta<\frac12\) and \(\beta<\alpha\), by Lemma~\ref{lem:D6}, $f^{(6)}(T)\to 0$.
Finally, \(\alpha<\frac12\) and Lemma~\ref{lem:D7} yield $f^{(7)}(T)\to 0$.
This proves the lemma.
\end{proof}

With Lemma \ref{lem:Difirelation}, \eqref{eq:FTsplitDis} turns into
\begin{equation}\label{eq:FTsplitfis}
    \left|\P_x(Z_T\leq y)-\frac{1}{\sqrt{\pi}}\frac{d_{x,\mu}(\overline{c})}{\sqrt{\psi''(\overline{c})}}T^{-1/2}e^{T\psi(c^\ast_T)}\right|
    \leq\left(g^{(1)}+\ldots+g^{(7)}\right)\,d_{x,\mu}(\overline{c})\,T^{-1/2}\,e^{T\psi(c^\ast_T)},
\end{equation}
for all $T\geq\max\left\{T^{(1)},\ldots,T^{(7)}\right\}$ and $x,y \text{ and }\mu$ satisfying the conditions mentioned in Theorem~\ref{thm:uniformasymptotics} for all such $T$.
As the $g^{(i)}$ individually vanish as $T\xrightarrow{}\infty$, their sum also vanishes. Therefore, for any fixed $\delta\in(0,1)$,
there exists some $T_{M,\delta,\beta,K}\geq \max\left\{T^{(1)},\ldots,T^{(7)}\right\}$ such that for all $T\geq T_{M,\delta,\beta, K}$
\begin{equation*}
    g^{(1)}+\ldots+g^{(7)}
    \leq \delta\,\frac{1}{\sqrt{\pi}}\frac{1}{\sqrt{\psi''(\overline{c})}},
\end{equation*}
and plugging this into \eqref{eq:FTsplitfis} yields
\begin{equation}\label{eq:resultnoteval1}
    \left(1-\delta\right)\frac{1}{\sqrt{\pi}}\,\frac{d_{x,\mu}(\overline{c})}{\sqrt{\psi''(\overline{c})}}\,T^{-1/2}\,e^{T\psi(c^\ast_T)}
    \leq \P_x(Z_T\leq y)
    \leq \left(1+\delta\right)\frac{1}{\sqrt{\pi}}\,\frac{d_{x,\mu}(\overline{c})}{\sqrt{\psi''(\overline{c})}}\,T^{-1/2}\,e^{T\psi(c^\ast_T)},
\end{equation}
for all $T\geq T_{M,\delta,\beta,K}$, for any $x,\mu$ where $ x^2+\mu^2\leq MT^{\beta}$ and any $y>0$ with $\abs y-\theta T\abs\leq KT^\beta$ for all $T\geq T_{M,\delta,\beta,K}$.\\

We shall now determine the rate of the explicit terms in the asymptotics.
\begin{lem}\label{lem:finalsimplification}
Under the assumptions of Theorem~\ref{thm:uniformasymptotics}, i.e. for any $x=x(T),\mu=\mu(T)$ and $y=y(T)>0$ that satisfy $\mu^2+x^2\le MT^\beta$ and $|y-\theta T|\le KT^\beta$ for sufficiently large $T$ we have uniformly, as \(T\to\infty\),
\begin{equation} \label{eqn:eqlem10}
\frac{1}{\sqrt{\pi}}
\frac{d_{x,\mu}(\overline c)}
{\sqrt{\psi''(\overline c)}}
T^{-1/2}
e^{T\psi(c_T^\ast)}
\sim
\frac{4\sqrt y}{\sqrt\pi\,T}
\exp\!\left(
-\frac{(T+x^2)^2}{8y}
-\frac{\mu^2T}{2}
-\mu x
+\mu^2\frac yT
\right).
\end{equation}
\end{lem}

\begin{proof}
    First we shall calculate the explicit form of the asymptotics. Plugging the definition
$
\overline c=\frac1{8\theta^2}
$ and $
g(\overline c)
=
\sqrt{2\overline c}
=
\frac1{2\theta}
$ into $d_{x,\mu}$ yields
\begin{align*}
d_{x,\mu}(\overline c)
&=
\frac{
\exp\!\left(
-\mu x
-\frac{2\overline c\,x^2-\mu^2}
     {2g(\overline c)}
\right)
}
{\overline c}=
8\theta^2
\exp\!\left(
-\mu x
-\theta\left(
\frac{x^2}{4\theta^2}
-\mu^2
\right)
\right)=
8\theta^2
\exp\!\left(
-\frac{x^2}{4\theta}
-\mu x
+\mu^2\theta
\right).
\end{align*}

Continuing, by \eqref{eq:psiderivatives} we have
$\psi''(\overline c)=4\theta^3$
and since
$g(c_T^\ast)=\frac1{2z}$,
we obtain
\begin{align*}
\psi(c_T^\ast)
&=
zc_T^\ast
-\frac{\mu^2}{2}
-\frac{g(c_T^\ast)}2=
z\frac1{8z^2}
-\frac{\mu^2}{2}
-\frac1{4z}=
-\frac1{8z}
-\frac{\mu^2}{2}.
\end{align*}

This gives us the expression
\begin{align*}
\frac{1}{\sqrt{\pi}}
\frac{d_{x,\mu}(\overline c)}
{\sqrt{\psi''(\overline c)}}
T^{-1/2}
e^{T\psi(c_T^\ast)}
&=
\frac{1}{\sqrt{\pi}}
\frac{
8\theta^2
}
{\sqrt{4\theta^3}}
\exp\!\left(
-\frac{x^2}{4\theta}
-\mu x
+\mu^2\theta
\right)
T^{-1/2}
\exp\!\left(
-\frac{T}{8z}
-\frac{\mu^2T}{2}
\right)
\\
&=
\frac{4\sqrt{\theta}}{\sqrt{\pi}}
\,T^{-1/2}
\exp\!\left(
-\frac{T}{8z}
-\frac{x^2}{4\theta}
-\frac{\mu^2T}{2}
-\mu x
+\mu^2\theta
\right).
\end{align*}

Using
$
|z-\theta|
\le KT^{\beta-1},
x^2+\mu^2\le MT^\beta,
$
and \eqref{eq:zthetabound} we obtain
\begin{align}\label{eq:termwiseboundsforexponent}
\left|\frac{x^2}{4\theta}
-
\frac{x^2}{4z}\right|
=
\frac{x^2}{4}\frac{|z-\theta|}{|z\theta|}
\leq
\frac{KMT^{2\beta-1}}{2\theta^2},
\qquad
\mu^2|z-\theta|
\leq KM
T^{2\beta-1} 
\quad\text{and}\quad
\frac{x^4}{8Tz}
\leq
\frac{M^2T^{2\beta-1}}{4\theta}.
\end{align}
Therefore, since
$
\frac{T}{8z}
+
\frac{x^2}{4z}
+
\frac{x^4}{8Tz}
=
\frac{(T+x^2)^2}{8Tz}$ and $
z=\frac yT,
$
\begin{align*}
&\left|\frac{\frac{1}{\sqrt{\pi}}
\frac{d_{x,\mu}(\overline c)}
{\sqrt{\psi''(\overline c)}}
T^{-1/2}
e^{T\psi(c_T^\ast)}
}{\frac{4\sqrt y}{\sqrt{\pi}\,T}
\exp\!\left(
-\frac{(T+x^2)^2}{8y}
-\frac{\mu^2T}{2}
-\mu x
+\mu^2\frac yT
\right)} \right| = \left|\frac{\frac{4\sqrt{\theta}}{\sqrt{\pi}}
\,T^{-1/2}
\exp\!\left(
-\frac{T}{8z}
-\frac{x^2}{4\theta}
-\frac{\mu^2T}{2}
-\mu x
+\mu^2\theta
\right)}{\frac{4\sqrt z}{\sqrt{\pi}}T^{-1/2}
\exp\!\left(
-\frac{T}{8z}
-
\frac{x^2}{4z}
-
\frac{x^4}{8Tz}
-\frac{\mu^2T}{2}
-\mu x
+\mu^2\frac yT
\right)}\right|\\
=~& \left|\sqrt{\frac{\theta}{z}}
\,
\exp\!\left(
-\left(\frac{x^2}{4\theta}
-\frac{x^2}{4z}\right)
+ \frac{x^4}{8Tz}
+\mu^2(\theta-z)
\right)\right|\longrightarrow 1,
\end{align*}
by \eqref{eq:termwiseboundsforexponent}.
This shows the claim. 
\end{proof}

Theorem~\ref{thm:uniformasymptotics} now follows from  \eqref{eq:resultnoteval1} and (\ref{eqn:eqlem10}).


\section{Uniform asymptotics for the ratio of the small ball probabilities}
\label{sec:unif_ratiosbp}
In this section, we will prove a similar uniform asymptotic relation for the ratio of two small ball probabilities.
These ratios will later arise when we use the Markov property to ``disconnect" the asymptotic behaviour of the conditioned process from its behaviour on finite time intervals.
We need this asymptotic formula to hold uniformly (over a range of starting values $w$ and $y_0$) since we want to use it not pointwise but in expectation.
\begin{prop}\label{prop:ratiouniformasymptotics}
Let $x\in\R$ and $t_0>0$ as well as $\theta>0$, $0<\beta<\frac{1}{2}$, $M,K>0$, $\delta\in(0,1)$ be fixed.
Then there is a $T_{M,\delta,\beta,K}>0$ such that for all $w=w(T),\mu=\mu(T)$ and $y=y(T),y_0=y_0(T)>0$ with
\[
\mu^2+w^2\le MT^\beta,
\qquad
|y-\theta T|\le KT^\beta
\qquad \text{and} \qquad 0\leq y_0\leq T^\beta,
\]
for all $T\geq T_{M,\delta,\beta,K}$, we have
\[
(1-\delta)\,h_\mu(w,y_0,t_0)
\leq
\frac{\P_w(Z_{T-t_0}\leq y-y_0)}
     {\P_x(Z_T\leq y)}
\leq
(1+\delta)\,h_\mu(w,y_0,t_0),
\]
for all $T\geq T_{M,\delta,\beta,K}$, where
\[
h_\mu(w,y_0,t_0)
:=
\exp\!\left(
\frac{x^2-w^2+t_0}{4\theta}
-\frac{y_0}{8\theta^2}
+\mu(x-w)
+\frac{\mu^2 t_0}{2}
\right).
\]
\end{prop}

\begin{proof}
    We prove only the lower bound, as this is the estimate needed in the sequel. The proof of the upper bound is entirely analogous. Since $\beta<\frac12$ and $t_0$ is fixed, there exists a $T_0>0$ such that $ t_0 \leq \frac{T^\beta}{2}$ and $T-t_0>\frac{T}{2}$ for every $T\geq T_0$.
    For any $w,\mu,y,y_0$ satisfying the assumptions of Proposition~\ref{prop:ratiouniformasymptotics} for $T\geq T_0$ this implies
    \begin{align}\label{eq:estimatey-y_0}
        \abs y-y_0-\theta (T-t_0)\abs&\leq \abs y-\theta T\abs+ y_0+\theta t_0
        \leq \left(K+1+\frac{\theta}{2}\right)T^\beta\leq  K'\,(T-t_0)^\beta,
    \end{align}
    where $K':= 2^\beta \left(K+1+\frac{\theta}{2}\right)$, and $ \mu^2+w^2\leq MT^\beta \leq 2^\beta M(T-t_0)^\beta =: M'(T-t_0)^\beta$ for all $T\geq T_0$. Hence, by Theorem~\ref{thm:uniformasymptotics}, there exists a
    $T^{(1)}_{M,\delta,\beta,K}\geq T_0$ such that for all $T\geq T^{(1)}_{M,\delta,\beta,K}$,
    \begin{multline}\label{eq:sbprationumerator}
        \P_w(Z_{T-t_0}\leq y-y_0)\\
        \geq (1-\delta/3)\,\frac{4\sqrt{y-y_0}}{\sqrt{\pi}(T-t_0)}
        \exp\left(
        -\frac{(T-t_0+w^2)^2}{8(y-y_0)}
        -\frac{\mu^2(T-t_0)}{2}
        -\mu w
        +\mu^2\frac{y-y_0}{T-t_0}
        \right),
    \end{multline}
     for any $w,\mu,y,y_0$
    satisfying the assumptions of Proposition
    \ref{prop:ratiouniformasymptotics} for $T\geq T^{(1)}_{M,\delta,\beta,K}$. Note that $\mu^2+x^2\leq MT^\beta+x^2\leq 2MT^\beta$ for large enough $T$ and hence, together with Theorem~\ref{thm:uniformasymptotics}, we get a
    $T^{(2)}_{M,\delta,\beta,K}\geq T_0$ such that for all $T\geq T^{(2)}_{M,\delta,\beta,K}$ 
    \begin{equation}\label{eq:sbpratiodenominator}
        \P_x(Z_T\leq y)
        \leq
        (1+\delta/3)\,
        \frac{4\sqrt y}{\sqrt{\pi}T}
        \exp\left(
        -\frac{(T+x^2)^2}{8y}
        -\frac{\mu^2T}{2}
        -\mu x
        +\mu^2\frac{y}{T}
        \right),
    \end{equation}
    for any $y>0$
    with $|y-\theta T|\leq KT^\beta$ and $\mu^2\leq MT^\beta$. Combining \eqref{eq:sbprationumerator} and
    \eqref{eq:sbpratiodenominator} therefore yields
    \begin{multline}
        \frac{\P_w(Z_{T-t_0}\leq y-y_0)}
             {\P_x(Z_T\leq y)}
        \geq
        \frac{1-\delta/3}{1+\delta/3}
        \frac{\sqrt{y-y_0}}{\sqrt y}
        \frac{T}{T-t_0}
        \\
        \times
        \exp\Bigg(
        -\frac{(T-t_0+w^2)^2}{8(y-y_0)}
        +\frac{(T+x^2)^2}{8y}
        +\frac{\mu^2t_0}{2}
        +\mu(x-w)
        +\mu^2\left(
        \frac{y-y_0}{T-t_0}
        -\frac{y}{T}
        \right)
        \Bigg),\nonumber
    \end{multline}
for all
$
T\geq
\max\left\{
T^{(1)}_{M,\delta,\beta,K},
T^{(2)}_{M,\delta,\beta,K}
\right\}
$ and any $w,\mu,y,y_0$ fulfilling the assumptions of Proposition
\ref{prop:ratiouniformasymptotics}.\\
We now analyze the right-hand side of the previous inequality.
Factoring out the candidate limit
$h_\mu(w,y_0,t_0)$
and using
$
\frac{T}{T-t_0}\ge1,
$
we obtain 
    \begin{align}\label{eq:sbpratioboundexact}
        \frac{\P_w(Z_{T-t_0}\leq y-y_0)}
             {\P_x(Z_T\leq y)}
        \geq
        \frac{1-\delta/3}{1+\delta/3}
        \,h_\mu(w,y_0,t_0)\,
        \exp(R),
    \end{align}
    where
    \begin{align*}
        R:=&
        -\frac{x^2-w^2+t_0}{4\theta}
        +\frac{y_0}{8\theta^2}
        -\frac{(T-t_0+w^2)^2}{8(y-y_0)}
        +\frac{(T+x^2)^2}{8y}
        +\mu^2\left(
        \frac{y-y_0}{T-t_0}
        -\frac{y}{T} 
        \right)+ \frac{1}{2}\ln \left( 1-\frac{y_0}{y} \right).
    \end{align*}
    It therefore remains to show that the remainder term
$R$ is uniformly small over the admissible range of parameters.
Indeed, if there exists
$T_{M,\delta,\beta,K}
\ge
\max\{T^{(1)}_{M,\delta,\beta,K},
T^{(2)}_{M,\delta,\beta,K}\}$
such that
\begin{equation}\label{eq:exponentialasymptoticsbound}
        \exp(R)\geq 1-\delta/3,
    \end{equation}
    for all $T\geq T_{M,\delta,\beta,K}$, then the claim follows, as
    \begin{equation*}
        \frac{(1-\delta/3)^2}{1+\delta/3}=\frac{(1+\delta/3)^2-4\delta/3}{1+\delta/3}
        = 1+\delta/3-\frac{4\delta/3}{1+\delta/3}
        \geq 1+\delta/3-4\delta/3=1-\delta
    \end{equation*}
    and inserting into \eqref{eq:sbpratioboundexact} yields
    \begin{align}\label{eq:sbpratiobound}
        \frac{\P_w(Z_{T-t_0}\leq y-y_0)}{\P_x(Z_T\leq y)}
        &\geq\frac{1-\delta/3}{1+\delta/3}\,\,h_\mu(w,y_0,t_0)\, (1-\delta/3)\geq (1-\delta)\,\,h_\mu(w,y_0,t_0)\,,
    \end{align}
    for all $T\geq T_{M,\delta,\beta,K}$ and any $w,\mu,y,y_0$ fulfilling the assumptions of Proposition
    \ref{prop:ratiouniformasymptotics}.\\
    First let $T_1\geq \max\{T^{(1)}_{M,\delta,\beta,K},T^{(2)}_{M,\delta,\beta,K},T_0, (\theta/4)^{1/(\beta-1)},1\} $ so that
    \begin{align*}
        &(a)\quad \frac{\theta T}{2} \leq \theta T - KT^\beta \leq y \leq \theta T + KT^\beta \leq \frac{3\theta T}{2}
        \qquad (b)\quad T-t_0\geq \frac{T}{2} \qquad (c) \quad t_0\leq \frac{T^\beta}{2}\\
        & (d) \quad y-y_0\geq \frac{\theta T}{2} \qquad (e)\quad x^2\leq MT^\beta.
    \end{align*}
We estimate the three contributions to $R$
separately.
For $T\geq T_1$, $y\ge \theta T/2>0$ and together with $y_0\le T^\beta$
 we obtain
\begin{equation}\label{eq:unif_asymp_eq1}
\left|
\frac12
\ln\!\left(1-\frac{y_0}{y}\right)
\right|
= -\frac12
\ln\!\left(1-\frac{y_0}{y}\right)
\le
\frac{y_0}{y}
\le
\frac2\theta T^{\beta-1}
\leq\frac{2}{\theta} T^{2\beta-1},
\end{equation}
where we used, that $\ln(1-x)\geq -2x$ on $(0,\frac{1}{2})$ and $\frac{y_0}{y}\leq \frac{2T^\beta}{\theta T}=\frac{2}{\theta}T^{\beta-1}\leq \frac12$ for $T\geq T_1$, as $T^{\beta-1}\leq\theta/4$. Since
$
\frac{y-y_0}{T-t_0}-\frac yT
=
\frac{yt_0-Ty_0}{T(T-t_0)},
$
(a)--(d), as well as $\mu^2\leq MT^\beta$ and $y_0\leq T^\beta$ imply
\begin{equation}\label{eq:unif_asymp_eq2}
\mu^2
\left|
\frac{y-y_0}{T-t_0}
-
\frac yT
\right|
\le
\frac{(3\theta+4)M}{2}
T^{2\beta-1}.
\end{equation}

For the remaining contribution, we first expand the quadratic terms to obtain
\begin{align}\label{eq:unif_asymp_eq3}
&\left|
\frac{(T+x^2)^2}{8y}
-
\frac{(T-t_0+w^2)^2}{8(y-y_0)}
-
\frac{x^2-w^2+t_0}{4\theta}
+
\frac{y_0}{8\theta^2}
\right|
\\
\le~&
\frac{x^2}{4}
\left|
\frac Ty-\frac1\theta
\right|
+
\frac{w^2}{4}
\left|
\frac{T-t_0}{y-y_0}-\frac1\theta
\right|
+
\frac18
\left|
\frac{x^4}{y}
-
\frac{w^4}{y-y_0}
\right|
+
\left|
\frac{T^2(y-y_0)-(T-t_0)^2y}
{8y(y-y_0)}
-
\frac{t_0}{4\theta}
+
\frac{y_0}{8\theta^2}
\right|. \notag
\end{align}

Each of the four terms on the right-hand side is estimated separately using
assumptions (a)--(e), as well as \eqref{eq:estimatey-y_0}.
More precisely,
\begin{align*}
\frac{x^2}{4}
\left|
\frac Ty-\frac1\theta
\right|
\le
\frac{MK}{2\theta^2}
T^{2\beta-1},\quad
\frac{w^2}{4}
\left|
\frac{T-t_0}{y-y_0}-\frac1\theta
\right|
\le
\frac{MK'}{2\theta^2}
T^{2\beta-1},\quad
\frac18
\left|
\frac{x^4}{y}
-
\frac{w^4}{y-y_0}
\right|
\le
\frac{M^2}{2\theta}
T^{2\beta-1},
\end{align*}
as well as
\begin{align*}
\left|
\frac{T^2(y-y_0)-(T-t_0)^2y}
{8y(y-y_0)}
-
\frac{t_0}{4\theta}
+
\frac{y_0}{8\theta^2}
\right|
&\le
\left(\frac{5K+3}{4\theta^3}+\frac{K+1+\frac{\theta}{4}}{4\theta^2}\right)T^{2\beta-1}.
\end{align*}
Indeed, we get the last estimate by multiplying out the term on the left hand side to get
\begin{align*}
    \left|\frac{T^2y-T^2y_0-T^2y+2Tt_0y-t_0^2y}{8y(y-y_0)}-\frac{t_0}{4\theta}+\frac{y_0}{8\theta^2}\right|
    =\left|y_0\left(\frac{1}{8\theta^2}-\frac{T^2}{8y(y-y_0)}\right)+t_0\left(\frac{2T-t_0}{8(y-y_0)}-\frac{1}{4\theta}\right)\right|
\end{align*}
and then combine the following estimates for the summands:
\begin{align*}
    &y_0\left|\frac{1}{8\theta^2}-\frac{T^2}{8y(y-y_0)}\right|
    =\frac{y_0}{8\theta^2}\left|\frac{y(y-y_0)-(\theta T)^2}{y(y-y_0)}\right|
    \leq\frac{y_0}{8\theta^2}\frac{4}{(\theta T)^2}\left(|y^2-(\theta T)^2|+| yy_0|\right)\\
    =~&\frac{y_0}{2\theta^4 T^2}\left(| y-\theta T|| y+\theta T|+|yy_0|\right)
    \leq \frac{y_0}{2\theta^4 T^2}\left(KT^\beta \frac{5\theta}{2}T+T^\beta \frac{3\theta}{2}T\right)
    \leq
    \frac{5K+3}{4\theta^3}T^{2\beta-1},
\end{align*}
and
\begin{align*}
    &t_0\left|\frac{2T-t_0}{8(y-y_0)}-\frac{1}{4\theta}\right|
    =t_0\left|\frac{T}{8(y-y_0)}-\frac{1}{8\theta}+\frac{T-t_0}{8(y-y_0)}-\frac{1}{8\theta}\right|\\
    \leq~& \frac{t_0}{8\theta}\frac{2}{\theta T}\left(\left|\theta T-(y-y_0)\right|+\left|\theta(T-t_0)-(y-y_0)\right|\right)
    \leq \frac{t_0}{4\theta^2 T}\left((K+1)T^\beta+\left(K+1+\frac{\theta}{2}\right)T^\beta\right)\\
    \leq~&\frac{1}{4\theta^2}\left(K+1+\frac{\theta}{4}\right)T^{2\beta-1}.
\end{align*}

Combining with \eqref{eq:unif_asymp_eq1}, \eqref{eq:unif_asymp_eq2} and \eqref{eq:unif_asymp_eq3} gives
\[
\left|
\frac{(T+x^2)^2}{8y}
-
\frac{(T-t_0+w^2)^2}{8(y-y_0)}
-
\frac{x^2-w^2+t_0}{4\theta}
+
\frac{y_0}{8\theta^2}
+\mu^2\left(
        \frac{y-y_0}{T-t_0}
        -\frac{y}{T} 
        \right)+ \frac{1}{2}\log \left( 1-\frac{y_0}{y} \right)
\right|
\le
CT^{2\beta-1}
\]
where
\[
C = \frac{MK}{2\theta^2}+\frac{MK'}{2\theta^2}+\frac{M^2}{2\theta}+\frac{5K+3}{4\theta^3}+\frac{K+1+\frac{\theta}{4}}{4\theta^2}+\frac{(3\theta+4)M}{2}
+\frac{2}{\theta}.
\]

Hence, $|R|\le C T^{2\beta-1}$
for every $T\ge T_1$.
Therefore, since $\beta<1/2$, there exists a $T_{M,\delta,\beta,K}\geq T_1$ such that
$\abs R\abs\leq \delta/(6e)<1/2$,
for all $T\geq T_{M,\delta,\beta,K}$. Lemma~18 in \cite{ouconditionedrestricted} then yields
\begin{equation*}
    \left|\exp\left(R\right)-1\right|\leq 2e\frac{\delta}{6e}=\frac{\delta}{3}
\end{equation*}
for all $T\geq T_{M,\delta,\beta,K}$, which shows that the condition \eqref{eq:exponentialasymptoticsbound} is indeed satisfied. As argued in \eqref{eq:sbpratioboundexact} and \eqref{eq:sbpratiobound}, this proves the lower bound.
\end{proof}

\section{Proof of Theorem~\ref{thm:conditionallimit}}
\label{sec:proof_main}
In this section, we will prove Theorem~\ref{thm:conditionallimit} by proving the following slightly more general version.
\begingroup 
\renewcommand{\thethm}{1'}
\begin{thm}\label{thm:conditionallimity}
Fix $x\in \R$ and $\theta>0$. Suppose that $\mu=\mu(T)$ satisfies $\mu^2 \leq M T^\beta$ for some $M>0$ and \(\beta<\frac{1}{2}\) and all sufficiently large $T$. Then, for any $y=y(T)$ satisfying $|y-\theta T|\leq KT^\beta$ for some $K>0$ and all sufficiently large $T$, as $T\to\infty$,
\[
\mathbb P_x\!\left(
(W_t)_{t\ge0}\in\cdot
\,\middle|\,
Z_T\le y
\right)
\]
converges weakly on $C([0,\infty))$ to the law of the Ornstein--Uhlenbeck process
$(X_t)_{t\ge0}$ satisfying
$$
X_0=x,
\qquad
dX_t
=
dB_t
-\frac{1}{2\theta}X_t\,dt,
\qquad t\ge0,
$$
where \((B_t)_{t\ge0}\) is a standard Brownian motion.
\end{thm}
\endgroup
Obviously, Theorem~\ref{thm:conditionallimit} is the special case of Theorem~\ref{thm:conditionallimity} with $y=\theta T$ and therefore a proof of the latter also proves the former theorem.
We stated our main result in Section~\ref{sec:introandresults} in its less general form since conditioning on $Z_T\leq \theta T$ appears more natural than the formulation of Theorem~\ref{thm:conditionallimity}, see the discussion beneath Theorem~\ref{thm:conditionallimit} in Section~\ref{sec:introandresults}.
However, proving the more general case of our main result does not introduce any additional complexity to the proof.\\

\textbf{Step 1: Bound on the Conditional Probability.}
First, fix a \(t_0>0\) and assume $\beta>0$ w.l.o.g as any $\mu$ and $y$ satisfying the assumptions of the theorem for some $\beta\leq 0$ also satisfy them for any $\beta\in(0,1/2)$ while the result does not depend on it. 
Define
\[G_m:=\left\{\omega\in C([0,t_0]):\omega_{t_0}^2\leq m^{\beta},\;Z_{t_0}(\omega)\leq m^{\beta}\right\}\quad\text{for all}~m\in\N.\]
Further, let $(\FF_t)_{t\geq 0}$ be the natural filtration induced by $W$ and equip $\CC([0,t_0])$ with the topology of uniform convergence and the Borel $\sigma$-algebra it induces (cf.\ \cite{whitt}).\\
Now consider any \(\delta\in(0,1)\) and let \(M>0\) such that $\mu^2\le MT^\beta$. Since \(\beta>0\), for every
\(T\ge m\) and every \(\omega\in G_m\), we have
$
\omega_{t_0}+\mu^2 \leq m^\beta + MT^\beta
\le
(1+M)T^\beta 
$
and
$
Z_{t_0}(\omega)\leq m^{\beta}
\le
T^\beta
$.
Therefore Proposition~\ref{prop:ratiouniformasymptotics} applied with
$
w=\omega_{t_0},
y_0=Z_{t_0}(\omega),
y=y_T
$
implies the existence of some $T_{M,\delta,\beta,K,m}>0$ such that, for all \(T\ge T_{M,\delta,\beta,K,m}\) and \(\omega\in G_m\),
\begin{equation}\label{eq:lowerbound_ratioconditioned}
\frac{
\P_{\omega_{t_0}}
\!\left(
Z_{T-t_0}
\le
y_T-Z_{t_0}(\omega)
\right)
}{
\P_x(Z_T\le y_T)
}
\ge
(1-\delta)
\exp\!\left(
\frac{x^2-\omega_{t_0}^2+t_0}{4\theta}
-\frac{Z_{t_0}(\omega)}{8\theta^2}
+\mu(x-\omega_{t_0})
+\frac{\mu^2t_0}{2}
\right).
\end{equation}
Let \(D\in\mathcal B(C([0,t_0]))\). 
By the Markov property,
\begin{align*}
    &\P_x\!\left(W\in D\,\middle|\,Z_T\le y_T\right)
    =\frac{1}{\P_x(Z_T\le y_T)}\,\P_x\!\left(W\in D,\,Z_T\le y_T\right)\\
    =~&\frac{1}{\P_x(Z_T\le y_T)}\,\E_x\!\left[\1_{W\in D}\,\,\P_{W_{t_0}}\!\left(Z_{T-t_0}'\leq y_T-Z_{t_0}\right)\right]
    =\E_x\!\left[\1_{W\in D}\,\,\frac{\P_{W_{t_0}}\!\left(Z_{T-t_0}'\leq y_T-Z_{t_0}\right)}{\P_x(Z_T\le y_T)}\right],
\end{align*}
where $Z_{T-t_0}':=\int_0^{T-t_0}{W'_s}^2 ds$ and $W'$ is a Brownian motion with drift $\mu$ starting at $W_{t_0}$ and the inner probability is taken only with respect to $W'$.
Restricting to \(D\cap G_m\), we may use inequality \eqref{eq:lowerbound_ratioconditioned} and obtain
\begin{align}
&\P_x\!\left(
W\in D
\,\middle| \,
Z_T\le y_T
\right)\ge
\E_x\!\left[
\mathbbm 1_{W\in D\cap G_m}\,\,
\frac{
\P_{W_{t_0}}
\!\left(
Z_{T-t_0}'
\le
y_T-Z_{t_0}(\omega)
\right)
}{
\P_x(Z_T\le y_T)
}
\right] \nonumber
\\
\ge~&
\E_x\!\left[
\mathbbm 1_{W\in D\cap G_m}\,\,
(1-\delta)
\exp\!\left(
\frac{x^2-W_{t_0}^2+t_0}{4\theta}
-\frac{Z_{t_0}}{8\theta^2}
+\mu(x-W_{t_0})
+\frac{\mu^2t_0}{2}
\right)
\right] \label{eq:expectation_wrt_martingale}
\end{align}
for all \(T\ge T_{M,\delta,\beta,K,m}\).
We want to interpret the exponential appearing in (\ref{eq:expectation_wrt_martingale}) as the value of a stochastic process at time \(t_0\). For this purpose, define
\[
M_t^{(T)}
:=
\exp\!\left(
\frac{x^2-W_{t}^2+t}{4\theta}
-\frac{Z_{t}}{8\theta^2}
+\mu(x-W_{t})
+\frac{\mu^2t}{2}
\right),
\qquad t,T\ge0.
\]
Then, we obtain that, for all \(T\ge T_{M,\delta,\beta,K,m}\), 
\begin{align*}
&
\P_x\!\left(
W\in D
\,\middle|\,
Z_T\le y_T
\right)\ge
(1-\delta)
\E_x\!\left[
\mathbbm 1_{W\in D\cap G_m}\,\,
M_{t_0}^{(T)}
\right].
\end{align*}

Fix $m\in\mathbb N$.
Taking the lower limit as $T\to\infty$ in the previous inequality yields
\begin{align}\label{eq:liminflowerboundP}
\liminf_{T\to\infty}
\P_x\!\left(
W\in D
\,\middle|\,
Z_T\le y_T
\right)
\ge
(1-\delta)
\liminf_{T\to\infty}
\E_x\!\left[
\mathbbm 1_{W\in D\cap G_m}\,\,
M_{t_0}^{(T)}
\right].
\end{align}

\textbf{Step 2: Establishing the Martingale Property.}
Applying It\^o's formula to
\(
W_t=x+\mu t+B_t
\),
we obtain
\[
x^2-W_t^2+t
=
-2\int_0^t W_s\,dB_s
-
2\mu\int_0^t W_s\,ds ,
\qquad \, t \geq 0.
\]

Since
$
W_t-x
=
B_t+\mu t,
$
substituting into the definition of $M_t^{(T)}$ yields
\begin{align*}
M_t^{(T)}
&=
\exp\!\Bigg(
-\frac1{2\theta}
\int_0^t W_s\,dB_s
-\frac{\mu}{2\theta}
\int_0^t W_s\,ds
-\frac1{8\theta^2}
\int_0^t W_s^2\,ds
-\mu B_t
-\frac{\mu^2t}{2}
\Bigg) 
\\
&=
\exp\!\Bigg(
-\int_0^t
\left(
\frac{W_s}{2\theta}
+
\mu
\right)
\,dB_s
-\frac12
\int_0^t
\left(
\frac{W_s}{2\theta}
+
\mu
\right)^2
\,ds
\Bigg),
\qquad\, t \geq 0.
\end{align*}

Define
$
a_T(t)
:=
\frac{W_t}{2\theta}
+\mu.
$
Then
\[
M_t^{(T)}
=
\exp\!\left(
-\int_0^t a_T(s)\,dB_s
-\frac12\int_0^t a_T(s)^2\,ds
\right) ,
\qquad \, t \geq 0.
\]

We first show that the stochastic exponential defining $M^{(T)}$ is a true martingale.
For this purpose, for $r\geq0$ define the local martingale
\[
Y_t^{(T,r)}
:=
-
\int_0^t
\mathbbm 1_{\{s\ge r\}}
a_T(s)\,dB_s,
\qquad
t\ge0.
\]
Since
$
\langle Y^{(T,r)}\rangle_t
=
\int_r^t
a_T(s)^2\,ds,
$
the stochastic exponential of $Y^{(T,r)}$ is given by
\[
\exp\left(
Y_t^{(T,r)}
-
\frac12
\langle Y^{(T,r)}\rangle_t
\right)
=
\frac{M_t^{(T)}}{M_r^{(T)}},
\qquad t\ge r.
\]

To verify Novikov's condition on a short time interval, let
$t>r$.
By the change of variables
$s=r+u(t-r)$,
Jensen's inequality, and Fubini's theorem, we obtain
\begin{align}\label{eq:novikovcondition}
&\mathbb E_x
\left[
\exp\left(
\frac12
\langle Y^{(T,r)}\rangle_t
\right)
\right] =
\mathbb E_x
\left[
\exp\left(
\frac12
\int_r^t
a_T(s)^2\,ds
\right)
\right]
\nonumber\\
=~&
\mathbb E_x
\left[
\exp\left(
\frac12
(t-r)
\int_0^1
a_T(r+u(t-r))^2\,du
\right)
\right]
\le\frac1{t-r}
\int_r^t
\mathbb E_x
\left[
\exp\left(
\frac12
(t-r)
a_T(s)^2
\right)
\right]
\,ds.
\end{align}

Now,
$
W_s=x+\mu s+B_s,
$
and hence
\[
a_T(s)
=
\frac{W_s}{2\theta}
+\mu
\sim
N\!\left(
\frac{x+\mu s}{2\theta}+\mu,
\,
\frac{s}{4\theta^2}
\right).
\]

Applying the moment generating function of the square of a Gaussian random
variable, we obtain
\begin{align}\label{eq:momentgenfunction}
\mathbb E_x
\left[
\exp\left(
\frac12
(t-r)a_T(s)^2
\right)
\right]
=
\left({1-\frac{(t-r)s}{4\theta^2}}\right)^{-1/2}\exp\!\left(
\displaystyle
\frac{\frac12(t-r)(\frac{x+\mu s}{2\theta}+\mu)^2}
{1-\frac{(t-r)s}{4\theta^2}}
\right).
\end{align}

This expression is finite whenever
$
\frac{(t-r)s}{4\theta^2}<1.
$
To guarantee this, choose a $\lambda\in(0,\frac{4\theta^2}{t_0})$.
Then for all $r<t\leq t_0$ with $t-r\leq\lambda$
we have
$\frac{(t-r)s}{4\theta^2}\leq \lambda \frac{t_0}{4\theta^2}<1$ for 
$s\in[r,t]$.
Consequently, the denominator in \eqref{eq:momentgenfunction} is uniformly bounded away from zero and the exponential is continuous on this compact interval. 
Therefore there
exists a finite constant $C>0$, not depending on $s$, such that
\[
\mathbb E_x
\left[
\exp\left(
\frac12
(t-r)a_T(s)^2
\right)
\right]
\le
C,
\qquad
s\in[r,t].
\]
Substituting this estimate into \eqref{eq:novikovcondition} yields
\[
\mathbb E_x
\left[
\exp\left(
\frac12
\langle Y^{(T,r)}\rangle_t
\right)
\right]
\le
\frac1{t-r}
\int_r^t
C\,ds
=
C
<
\infty.
\]
Thus Novikov's condition holds on every interval
$[r,r+\lambda]\cap[0,t_0]$, which shows that
$
\left(
M_t^{(T)}/M_r^{(T)}
\right)
$
is a martingale on these intervals.
Since the stochastic exponential of a continuous local martingale is itself
a nonnegative local martingale, $(M_t^{(T)})_{t\ge0}$ is a supermartingale.
To conclude that it is a true martingale, it therefore suffices to show that
its expectation remains constant.
Now let $t\in[n\lambda,(n+1)\lambda]\cap[0,t_0]$. Since $M^{(T)}$ is adapted and
$\bigl(M_t^{(T)}/M_{n\lambda}^{(T)}\bigr)_{t\in[n\lambda,\min\{(n+1)\lambda,t_0\}]}$
is a martingale,
\begin{align*}
\mathbb E_x\left[M_t^{(T)}\right]=
\mathbb E_x
\left[
\mathbb E_x
\left[
M_{n\lambda}^{(T)}
\frac{M_t^{(T)}}{M_{n\lambda}^{(T)}}
\,\middle|\,
\FF_{n\lambda}
\right]
\right]
=
\mathbb E_x
\left[
M_{n\lambda}^{(T)}
\,
\mathbb E_x
\left[
\frac{M_t^{(T)}}{M_{n\lambda}^{(T)}}
\,\middle|\,
\FF_{n\lambda}
\right]
\right]
=
\mathbb E_x\left[M_{n\lambda}^{(T)}\right].
\end{align*}
Iterating this identity over the finitely many intervals of size $\lambda$ covering $[0,t]$
shows that
\[
\mathbb E_x\left[M_t^{(T)}\right]
=
\mathbb E_x\left[M_0^{(T)}\right]
=
1,
\qquad
\forall t\in[0,t_0].
\]
It follows that $\left(M_t^{(T)}\right)_{t\in[0,t_0]}$ is a supermartingale with constant expectation, i.e. a martingale.\\
\textbf{Step 3: Application of Girsanov's Theorem.}
As $\left(M_t^{(T)}\right)_{t\in[0,t_0]}$ is a non-negative martingale with expectation one, for every $T>0$, the formula
$
dQ_T
=
M_{t_0}^{(T)}\,d\P_x
$
defines a probability measure $Q_T$ on
$(\Omega,\mathcal F_{t_0})$. Furthermore, by Girsanov's theorem, the process
\[
\widetilde B_t^{(T)}
:=
B_t
+
\int_0^t
a_T(s)\,ds
=
B_t
+
\int_0^t\left(\frac{W_s}{2\theta}+\mu\right)ds ,
\qquad \, t\in [0,t_0].
\]
is a Brownian motion under $Q_T$. Using
$
dW_t
=
dB_t+\mu\,dt,
$
we obtain
\begin{align*}
dW_t
=
d\widetilde B_t^{(T)}
-
\left(
\frac{W_t}{2\theta}
+
\mu
\right)dt
+
\mu\,dt
=
d\widetilde B_t^{(T)}
-
\frac{W_t}{2\theta}\,dt ,
\qquad \, t \in [0,t_0].
\end{align*}

Therefore, under $Q_T$, the process $W$ satisfies
\[
W_0=x,
\qquad dW_t
=
d\widetilde B_t^{(T)}
-
\frac{1}{2\theta}W_t\,dt,
\qquad \, t\in [0,t_0],
\]
for every $T>0$.
Since this SDE admits a unique strong solution,
the law of $W$ under $Q_T$ is identical for every $T>0$.
We denote this common law on $C([0,t_0])$ by $Q$. Hence,
$
Q_T(W\in D\cap G_m)
=
Q(D\cap G_m)
$
for every $T>0$. Using the definition of $Q_T$, we therefore obtain
$
\E_x\!\left[
\mathbbm 1_{W\in D\cap G_m}\,\,
M_{t_0}^{(T)}
\right]
=
Q(D\cap G_m),
$
and consequently from \eqref{eq:liminflowerboundP}
\[
\liminf_{T\to\infty}
\P_x\!\left(
W\in D
\,\middle|\,
Z_T\le y_T
\right)
\ge
(1-\delta)\,
Q(D\cap G_m).
\]

Since \(\delta>0\) was arbitrary, this yields
\begin{equation*}
    \liminf_{T\to\infty}
\P_x\!\left(
W\in D
\,\middle|\,
Z_T\le y_T
\right)
\ge
Q(D\cap G_m).
\end{equation*}

\textbf{Step 4: Convergence on $[0,t_0]$ for all $t_0>0$.}
By definition, the sequence \((G_m)_{m\in\N}\) is increasing and
\(\bigcup_{m=1}^{\infty}G_m=\CC([0,t_0]).\)
Therefore,
\begin{equation}\label{eq:lowerbound_probability}
    \liminf_{T\to\infty} \P_x\!\left(
W\in D
\,\middle|\,
Z_T\le y_T
\right)
\ge
\lim_{m\to\infty}
Q(D\cap G_m)
=
Q(D). 
\end{equation}

Since \(D^c\in\mathcal B(C([0,t_0]))\) as well, the same reasoning applied to
\(D^c\) yields
\begin{align}\label{eq:upperbound_probability}
  &\liminf_{T\to\infty}
\P_x\!\left(
W\in D^c
\,\middle|\,
Z_T\le y_T
\right)
\ge 
Q(D^c)
\quad\text{i.e.}\quad
 \limsup_{T\to\infty}
\P_x\!\left(
W\in D
\,\middle|\,
Z_T\le y_T
\right)
 \le
Q(D).
\end{align}

Therefore, (\ref{eq:lowerbound_probability}) combined with (\ref{eq:upperbound_probability}) shows that
\[
\lim_{T\to\infty}
\P_x\!\left(
W\in D
\,\middle|\,
Z_T\le y_T
\right)
=
Q(D).
\]
\textbf{Step 5: Extension to $[0,\infty)$.}
Let \(P^\ast\) denote the law on \(C([0,\infty))\) (or its restriction to $\CC([0,t_0])$) of an Ornstein--Uhlenbeck process \(X\), i.e. the unique strong solution with respect to some Brownian motion $B$ of (\ref{eqn:defnou}).

By the previous argument, for every \(D\in\BB(\CC([0,t_0]))\),
$
\lim_{T\to\infty}
\P_x\!\left(
(W_t)_{t\in[0,t_0]}\in D
\,\middle|\,
Z_T\le y_T
\right)
=
P^\ast\!\left(D\right).
$
This implies
$
\P_x\!\left(
(W_t)_{t\in[0,t_0]}\in\,\cdot\,
\,\middle|\,
Z_T\le y_T
\right)
\overset{w}{\longrightarrow}
P^\ast
$
on \(C([0,t_0])\).
Since \(t_0>0\) was arbitrary, this convergence holds for every finite time horizon and
the restrictions of $P^\ast$ to these finite time horizons are consistent. Therefore the conditional laws
$
\P_x\!\left(
(W_t)_{t\ge0}\in\cdot
\,\middle|\,
Z_T\le y_T
\right)
$
converge weakly to \(P^\ast\) on \(C([0,\infty))\) by Theorem 5 of \cite{whitt}. 
This completes the proof of Theorem~\ref{thm:conditionallimity}, which immediately proves our main result as discussed above.
\qed

\section{Proof of Theorem~\ref{thm:exponential}}
\label{sec:proof_exponential}
The proof is divided into three steps. 
The first one is a representation of the conditional law (Brownian motion with drift conditioned on $Z_T\leq \eta$) in terms of the law of the limiting Ornstein--Uhlenbeck process. 
In a second step, we show the $L^1$-convergence of the involved densities. 
Finally, we establish the weak convergence on $\CC([0,t_0])$ for any $t_0>0$, by actually showing convergence in total variation on these finite time-horizons. 
The weak convergence on \(C([0,\infty))\) then follows again from Theorem 5 in \cite{whitt}.\\


\textbf{{Step 1: A representation of the conditional law.}}
Fix
$
0=t_0<t_1<\cdots<t_d<T,
\, y_0=x,
$
and let \(C_1,\ldots,C_d\in\mathcal B(\mathbb R)\). Set
$
A:=\{W_{t_1}\in C_1,\ldots,W_{t_d}\in C_d\}.
$
By the Markov property of \(W\) and the memoryless property of the
exponential random variable \(\eta\),
\begin{align}\label{eq:markovandmemoryless}
\mathbb P_x(A\,|\,Z_T\le\eta)=
\int_{\mathbb R^d}
\mathbbm1_{C_1\times\cdots\times C_d}(y_1,\dots,y_d)
\prod_{i=1}^{d}
p_{\Delta t_i}^{\gamma,\mu}(y_{i-1},y_i)
\,
\frac{\mathbb P_{y_d}(Z_{T-t_d}\le\eta)}{\mathbb P_x(Z_T\le\eta)}
\,dy_1\cdots dy_d,
\end{align}
where $\Delta t_i:= t_i-t_{i-1}$ for all $i=1,\dots,d$ and    $p_t^{\gamma,\mu}(x,y)\,dy:=\P_x\left(W_t\in dy,\,Z_t\leq \eta\right)$.
We now use the explicit formulas for the killed semigroup and the survival probability. Since $\eta$ is an exponential random variable with mean $\frac{2}{\gamma^2}$, independent of $W$, by formulas~2.1.9.3 and 2.1.9.7 of \cite{borodin2002handbook}, we have
\begin{align}
\mathbb P_x(Z_T\le\eta)
&=
\mathbb E_x\!\left[
\exp\!\left(
-\frac{\gamma^2}{2}
\int_0^T W_s^2\,ds
\right)
\right]\nonumber\\
&=
\frac{1}{\sqrt{\cosh(\gamma T)}}
\exp\!\left(
-\mu x-\frac{\mu^2T}{2}-\frac{(x^2\gamma^2-\mu^2)}{2\gamma}\tanh(\gamma T)+
\frac{\mu x}{\cosh(\gamma T)}
\right),
\label{eq:htransform}\\
p_t^{\gamma,\mu}(x,y)\,dy&=\mathbb E_x\!\left[
\exp\!\left(
-\frac{\gamma^2}{2}
\int_0^t W_s^2\,ds
\right)
\mathbf 1_{\{W_t\in dy\}}
\right]\nonumber
\\
&=
\frac{\sqrt{\gamma}}
{\sqrt{2\pi\sinh(\gamma t)}}
\exp\!\left(
\mu(y-x)
-
\frac{\mu^2t}{2}
-
\frac{(x^2+y^2)\gamma\cosh(\gamma t)-2\gamma xy}
{2\sinh(\gamma t)}
\right)
\,dy.
\label{eq:transitiondensity}
\end{align}

We now simplify the two factors appearing in the integral representation \eqref{eq:markovandmemoryless}. 
Using the definitions of the hyperbolic functions,
we get the relations
\eqref{eq:transitiondensity}, we have 
\begin{align*}
    p_{t}^{\gamma,\mu}(x,y)
    &=\frac{\sqrt{\gamma}\,e^{-\frac{\mu^2}{2}t}}{\sqrt{2\pi\sinh(\gamma t)}}\,\frac{e^{\mu y+\frac{\gamma}{2}y^2}}{e^{\mu x+\frac{\gamma}{2}x^2}}\\
    &\qquad\qquad\times\exp\left(-\gamma\left(\frac{x^2}{2}\left(\frac{\cosh(\gamma t)}{\sinh(\gamma t)}-1\right)-\frac{2xy}{2\sinh(\gamma t)}+\frac{y^2}{2}\left(\frac{\cosh(\gamma t)}{\sinh(\gamma t)}+1\right)\right)\right)\\
    &=\frac{\sqrt{\gamma}\,e^{-\frac{\gamma+\mu^2}{2}t}}{\sqrt{\pi\left(1-e^{-2\gamma t}\right)}}\,\frac{e^{\mu y+\frac{\gamma}{2}y^2}}{e^{\mu x+\frac{\gamma}{2}x^2}}
    \,\exp\left(-\gamma\left(x^2\frac{e^{-2\gamma t}}{1-e^{-2\gamma t}}-\frac{2xye^{-\gamma t}}{1-e^{-2\gamma t}}+y^2\frac{1}{1-e^{-2\gamma t}}\right)\right)\\
    &=\frac{\sqrt{\gamma}}{\sqrt{\pi\left(1-e^{-2\gamma t}\right)}}\,\exp\left(-\frac{\gamma}{1-e^{-2\gamma t}}\left(y-x\,e^{-\gamma t}\right)^2\right)\,e^{-\frac{\gamma+\mu^2}{2}t}\,\frac{e^{\mu y+\frac{\gamma}{2}y^2}}{e^{\mu x+\frac{\gamma}{2}x^2}}\\
    &=: \widetilde{p}_t(x,y)\,e^{-\frac{\gamma+\mu^2}{2}t}\,\frac{e^{\mu y+\frac{\gamma}{2}y^2}}{e^{\mu x+\frac{\gamma}{2}x^2}}.
\end{align*}
We note that $\widetilde{p}$ is the transition kernel of an Ornstein--Uhlenbeck process with parameter $\gamma$.
Taking the product of the transition densities in \eqref{eq:markovandmemoryless} then yields 
\begin{align*}
    \prod_{i=1}^d\,p_{\Delta t_i}^{\gamma,\mu}(y_{i-1},y_i)
    &=\prod_{i=1}^d \left(\widetilde{p}_{\Delta t_i}(y_{i-1},y_i)\,e^{-\frac{\gamma+\mu^2}{2}\Delta t_i}\,\frac{e^{\mu y_i+\frac{\gamma}{2}y_i^2}}{e^{\mu y_{i-1}+\frac{\gamma}{2}y_{i-1}^2}}\right)\\
    &=\left(\prod_{i=1}^d\widetilde{p}_{\Delta t_i}(y_{i-1},y_i)\right)\exp\left(-\frac{\gamma+\mu^2}{2}t_d\right)\frac{\exp\left(\mu y_d+\frac{\gamma}{2}y_d^2\right)}{\exp\left(\mu x+\frac{\gamma}{2}x^2\right)}.
\end{align*}
Therefore, with
\begin{align*}
   g_T(\textbf{y}):=\left(\prod_{i=1}^d\widetilde{p}_{\Delta t_i}(y_{i-1},y_i)\right)\Psi_{T,t_d}(y_d),\qquad  \Psi_{T,t}(y):= e^{-\frac{\gamma+\mu^2}{2}t}\,\frac{\exp\left(\mu y+\frac{\gamma}{2}y^2\right)}{\exp\left(\mu x+\frac{\gamma}{2}x^2\right)} \frac{\P_{y}(Z_{T-t}\leq\eta)}{\P_x(Z_T\leq\eta)}.
\end{align*}
the conditional finite-dimensional distributions in \eqref{eq:markovandmemoryless} admit the representation
\begin{align}\label{eq:integralrep}
\P_x\!\left(W_{t_1}\in C_1,\ldots,W_{t_d}\in C_d\,\middle|\,Z_T\leq\eta\right)
=\int_{\R^d} \1_{C_1\times\cdots\times C_d}(\textbf{y})g_T(\textbf{y})\,d\textbf{y}.
\end{align}

\textbf{Step 2: Convergence of densities.}
We will now analyze the asymptotic behaviour of 
$g_T$.
Note that the $T$-dependent terms of $g_T(\textbf{y})$ are all contained in $\Psi_{T,t_d}(y_d)$.
Therefore, we examine $\Psi_{T,t}(y)$ for any $t\in[0,T)$ and $y\in\R$.
Applying \eqref{eq:htransform} to the numerator and denominator of the ratio of survival probabilities yields 
\begin{align*}
    \frac{\P_y(Z_{T-t}\leq\eta)}{\P_x(Z_T\leq\eta)}
    =\frac{\sqrt{\cosh(\gamma T)}}{\sqrt{\cosh(\gamma(T-t))}}\,\frac{\exp\left(-\mu y-\frac{\mu^2(T-t)}{2}-\frac{(y^2\gamma^2-\mu^2)}{2\gamma}\tanh(\gamma (T-t))+\frac{\mu y}{\cosh(\gamma (T-t))}\right)}
    {\exp\left(-\mu x-\frac{\mu^2T}{2}-\frac{(x^2\gamma^2-\mu^2)}{2\gamma}\tanh(\gamma T)+\frac{\mu x}{\cosh(\gamma T)}\right)}.
\end{align*}
Multiplying by $\exp\left(-\frac{\gamma+\mu^2}{2}t\right)\frac{\exp\left(\mu y+\frac{\gamma}{2}y^2\right)}{\exp\left(\mu x+\frac{\gamma}{2}x^2\right)}$ and simplifying leads to
\begin{align}\label{eq:explicitPsi}
    \Psi_{T,t}(y)&=e^{-\frac{\gamma}{2}t}\frac{\sqrt{\cosh(\gamma T)}}{\sqrt{\cosh(\gamma(T-t))}}\,\frac{\exp\left(\frac{\gamma}{2}y^2\left(1-\tanh(\gamma (T-t))\right)+\frac{\mu^2}{2\gamma}\tanh(\gamma(T-t))+\frac{\mu y}{\cosh(\gamma (T-t))}\right)}
    {\exp\left(\frac{\gamma}{2}x^2\left(1-\tanh(\gamma T)\right)+\frac{\mu^2}{2\gamma}\tanh(\gamma T)+\frac{\mu x}{\cosh(\gamma T)}\right)}.
\end{align}
We now determine the pointwise limit of $\Psi_{T,t}$ by considering the different factors in \eqref{eq:explicitPsi} separately.
We begin with the prefactor. Using the exponential representation of
the hyperbolic cosine,
\begin{align}\label{eq:Psilim1}
    \frac{\sqrt{\cosh(\gamma T)}}{\sqrt{\cosh(\gamma(T-t))}}
    =\sqrt{\frac{e^{\gamma T}(1+e^{-2\gamma T})}{e^{\gamma (T-t)}\left(1+e^{-2\gamma (T-t)}\right)}}
    =e^{\frac{\gamma}{2}t}\sqrt{\frac{1+e^{-2\gamma T}}{1+e^{-2\gamma (T-t)}}}
    \longrightarrow e^{\frac{\gamma}{2}t},
    \quad \text{for}~T\to\infty.
\end{align}
On the other hand, since
$\tanh(\gamma(T-t))\to 1$ and $\tanh(\gamma T)\to 1$ as $T\to \infty$,
we have
\begin{align}\label{eq:Psilim2}
    \exp\left(\frac{\gamma}{2}y^2\left(1-\tanh(\gamma (T-t))\right)-\frac{\gamma}{2}x^2\left(1-\tanh(\gamma T)\right)\right)\longrightarrow 1\quad \text{for}~T\to\infty.
\end{align}
Next, we will show that all drift-dependent terms vanish.
The two terms which are linear in \(\mu\) are controlled directly by the exponential
decay of the hyperbolic cosine: Since
$\frac{1}{\cosh(u)}\leq 2e^{-u}$ for all $u\ge 0$,
our assumption that
$|\mu|e^{-\gamma T}\to 0$ implies
\begin{align*}
    \left|\frac{\mu y}{\cosh(\gamma(T-t))}\right|\leq 2|y|\,|\mu|e^{-\gamma(T-t)}\longrightarrow 0
    \quad\text{and}\quad
    \left|\frac{\mu x}{\cosh(\gamma T)}\right|\leq 2|x|\,|\mu|e^{-\gamma T}\longrightarrow 0,
\end{align*}
which in turn implies
\begin{align}\label{eq:Psilim3}
    \exp\left(\frac{\mu y}{\cosh(\gamma (T-t))}-\frac{\mu x}{\cosh(\gamma T)}\right)\longrightarrow 1\quad \text{for}~T\to\infty.
\end{align}
For the remaining term, since $\tanh(a)-\tanh(b)=\frac{\sinh(a-b)}{\cosh(a)\cosh(b)}$ and the $\sinh$ is odd,
we obtain
\begin{align*}
    \left|\frac{\mu^2}{2\gamma}\left(\tanh(\gamma(T-t))-\tanh(\gamma T)\right)\right|
    =\frac{\mu^2}{2\gamma}\left|\frac{\sinh(-\gamma t)}{\cosh(\gamma(T-t))\cosh(\gamma T)}\right|
    \leq \frac{2}{\gamma}|\sinh(\gamma t)|e^{\gamma t}\,\mu^2e^{-2\gamma T}\longrightarrow 0,
\end{align*}
by the assumption $|\mu|e^{-\gamma T}\to 0$; and therefore
\begin{align}\label{eq:Psilim4}
    \exp\left(\frac{\mu^2}{2\gamma}\left(\tanh(\gamma(T-t))-\tanh(\gamma T)\right)\right)\longrightarrow 1,\quad \text{for}~T\to\infty. 
\end{align}
Combining \eqref{eq:explicitPsi} with the preceding limits \eqref{eq:Psilim1}, \eqref{eq:Psilim2}, \eqref{eq:Psilim3} and \eqref{eq:Psilim4}, for every fixed $t\geq 0$ and $y\in\R$
\begin{align*}
    \Psi_{T,t}(y)\longrightarrow e^{-\frac{\gamma}{2}t}\cdot e^{\frac{\gamma}{2}t}=1,\quad\text{for}~T\to\infty.
\end{align*}
Thus, for every $\textbf{y}\in\R^d$, we have
\begin{align}\label{eq:pointwiseconvergencegT}
    g_T(\textbf{y})=\left(\prod_{i=1}^d\widetilde{p}_{\Delta t_i}(y_{i-1},y_i)\right)\Psi_{T,t_d}(y_d)\longrightarrow \prod_{i=1}^d\widetilde{p}_{\Delta t_i}(y_{i-1},y_i):=g(\textbf{y})\quad\text{for}~T\to\infty.
\end{align}

We now note that both $g_T$ and $g$ are probability densities on $\R^d$:
Both functions are clearly non-negative by definition and setting $C_i=\R$ for each $i=1,\ldots,d$ in \eqref{eq:integralrep} yields $\int_{\R^d}g_T(\textbf{y})\,d\textbf{y}=1$.
Since $\widetilde{p}$ is the transition kernel of an Ornstein--Uhlenbeck process $U$ with parameter $\gamma$, analogously to \eqref{eq:markovandmemoryless} we have
\begin{align}\label{eq:OUmarkovandmemoryless}
    \P_x\!\left(U_{t_1}\in C_1,\ldots,U_{t_d}\in C_d\right)
    =\int_{\R^d}\1_{C_1\times\cdots\times C_d}(\textbf{y})\,\prod_{i=1}^d\widetilde{p}_{\Delta t_i}(y_{i-1},y_i)\,d\textbf{y}
    =\int_{\R^d}\1_{C_1\times\cdots\times C_d}(\textbf{y})\,g(\textbf{y})\,d\textbf{y}.
\end{align}
Again setting $C_i=\R$ for each $i=1,\ldots,d$ yields
$\int_{\R^d}\,g(\textbf{y})\,d\textbf{y}=1$. 
Therefore, together with \eqref{eq:pointwiseconvergencegT}, we have shown that the $g_T$ are probability densities that converge pointwise to the probability density $g$.
Scheffé's Lemma then implies that these densities also converge in $L^1$, i.e.
\begin{align}\label{eq:gTL1convergence}
    \int_{\R^d}|g_T(\textbf{y})-g(\textbf{y})|\,d\textbf{y}\longrightarrow 0.
\end{align}

\textbf{Step 3: Weak Convergence.}
Fix \(t_0>0\). For \(T>t_0\), let
\[
Q_T^x(A):=\P_x\!\left((W_t)_{0\le t\le t_0}\in A\,\middle|\,Z_T\le\eta\right)
\quad\text{and}\quad
Q^{x}(A):=\P_x(U\in A)\quad \text{for}~A\in\BB\left(\CC([0,t_0])\right).
\]
We shall use the $L^1$-convergence of the densities shown in Step~2 to show the claimed weak convergence of the family of measures $\left(Q_T^x\right)_{T>t_0}$ on \(\CC([0,t_0])\).
First, we will show an alternative representation of $Q_T^x$. 
Consider $0=s_0<s_1<\ldots<s_d=t_0$ and any $C_1,\ldots,C_d\in\BB(\R)$. 
Further, set $A:=\left\{f\in\CC([0,t_0]):f_{s_1}\in C_1,\ldots,f_{s_d}\in C_d\right\}$ and $\Delta s_i:=s_i-s_{i-1}$.
Then, by \eqref{eq:integralrep} and \eqref{eq:OUmarkovandmemoryless} we have
\begin{align*}
    Q_T^x(A)&=\P_x\!\left(W_{s_1}\in C_1,\ldots,W_{s_d}\in C_d\middle|\,Z_T\leq\eta\right)
    =\int_{\R^d}\1_{C_1\times\cdots\times C_d}(\textbf{y})\left(\prod_{i=1}^d\widetilde{p}_{\Delta s_i}(y_{i-1},y_i)\right)\Psi_{T,s_d}(y_d)\,d\textbf{y}\\
    &
    =\E_x\!\left[\1_{U\in A}\Psi_{T,s_d}(U_{s_d})\right].
\end{align*}
Since the finite point evaluations generate $\BB\left(\CC([0,t_0])\right)$ (cf.\ Theorem 2 in \cite{whitt}) and the two measures 
agree on this generator, it holds
\begin{align}\label{eq:QTrep}
    Q_T^x(A)=\E_x\!\left[\1_{U\in A}\Psi_{T,t_0}(U_{t_0})\right]\quad\text{for all}~A\in\BB\left(\CC([0,t_0])\right).
\end{align}
From \eqref{eq:gTL1convergence} (with $d=1$) we get
\begin{align}\label{eq:PsiL1convergence}
    \E_x\!\left[\left|\Psi_{T,t_0}(U_{t_0})-1\right|\right]
    =\int_{\R}\widetilde{p}_{t_0}(x,y)|\Psi_{T,t_0}(y)-1|\,dy
    =\int_{\R}|g_T(y)-g(y)|\,dy\longrightarrow 0.
\end{align}
Combining \eqref{eq:QTrep} and \eqref{eq:PsiL1convergence} yields
\begin{align*}
    \sup_{A\in\BB\left(\CC([0,t_0])\right)}\left|Q_T^x(A)-Q^x(A)\right|
    \leq\sup_{A\in\BB\left(\CC([0,t_0])\right)}\E_x\!\left[\1_{U\in A}\left|\Psi_{T,t_0}(U_{t_0})-1\right|\right]
    \leq\E_x\!\left[\left|\Psi_{T,t_0}(U_{t_0})-1\right|\right]\to 0.
\end{align*}
This shows that $\left(Q_T^x\right)_{T>t_0}$ converges in total variation to $Q^x$ on $\BB\left(\CC([0,t_0])\right)$, which immediately implies their weak convergence.
As the weak convergence of those measures holds on $\CC([0,t_0])$ for any $t_0>0$, Theorem 5 in \cite{whitt} extends it to $\CC([0,\infty))$, which was the claim.\qed

\bibliographystyle{alpha}
\bibliography{references}


\appendix

\section{Auxiliary identities}
\label{sec:auxiliary1}
Here, we collect some inequalities that are used regularly.
\begin{lem}\label{lem:generalbounds}
    Let $u=c+iv$ with $c>0$, $v\in \R$ and $g(u)=\sqrt{2u}$ (principal branch of the complex square root). We collect some useful bounds:
    \begin{enumerate}[(a)]
        \item $\Re g(u)\geq g(c)=\sqrt{2c}$, \label{itm:a}
        \item $\Re \frac{1}{g(u)}\leq   \frac{1}{g(c)} = \frac{1}{\sqrt{2c}}$, \label{itm:b}
        \item $\left| \frac{1}{1+e^{-2Tg(u)}} \right|\leq 2$ for $T\geq \frac{\ln2}{2\sqrt{2c}}$, \label{itm:c}
        \item $\frac{1}{\sqrt{|u|}} \leq \frac{1}{\sqrt{c}}$,\label{itm:d}
        \item $ \sqrt{|u|} \leq \Re \sqrt{2u}$,\label{itm:e}
    \end{enumerate}
\end{lem}
\begin{proof}
If $u=c+iv$ such that $c>0$, then the square root of $u$ is given by
\[
\sqrt{u}=\sqrt{\frac{\sqrt{c^2+v^2}+c}{2}}+i \text{ sgn}(v)\sqrt{\frac{\sqrt{c^2+v^2}-c}{2}}
\]

The inequality (\ref{itm:a}) comes directly from the fact that $\Re g(u) = \sqrt{2}\Re \sqrt{u} = \sqrt{{\sqrt{c^2+v^2}+c}}\geq\sqrt{2c}$.

For (\ref{itm:b}), using the form of $\sqrt{2u}$, we get
\[
\frac1{\sqrt{2u}}
=
\frac{
\sqrt{\sqrt{c^2+v^2}+c}
-
i\,\operatorname{sgn}(v)\sqrt{\sqrt{c^2+v^2}-c}
}{
(\sqrt{\sqrt{c^2+v^2}+c})^2+(\sqrt{\sqrt{c^2+v^2}-c})^2
} = \frac{
\sqrt{\sqrt{c^2+v^2}+c}
-
i\,\operatorname{sgn}(v)\sqrt{\sqrt{c^2+v^2}-c}
}{
 2\sqrt{c^2+v^2}
},
\]
so that
\[
\Re\left(\frac1{\sqrt{2u}}\right)
=
\frac{\sqrt{\sqrt{c^2+v^2}+c}}{2\sqrt{c^2+v^2}}\leq \frac{\sqrt{2\sqrt{c^2+v^2}}}{2\sqrt{c^2+v^2}}=\frac{1}{\sqrt{2\sqrt{c^2+v^2}}}\leq\frac{1}{\sqrt{2c}}.
\]

For (\ref{itm:c}), note that 
\[\left| \frac{1}{1+e^{-2Tg(u)}} \right|\leq \frac{1}{1-|e^{-2Tg(u)}|}= \frac{1}{1-e^{-2T\Re g(u)}}\leq \frac{1}{1-e^{-2T\sqrt{2c}}}.\]

And since $e^{-2T\sqrt{2c}} \leq 1/2$ for $T\geq\frac{\ln 2}{2\sqrt{2c}}$, we have $\left| \frac{1}{1+e^{-2Tg(u)}} \right|\leq 2$.

For (\ref{itm:d}), observe that
$\frac{1}{\sqrt{|u|}}=\frac{1}{\sqrt{\sqrt{c^2+v^2}}}\leq \frac{1}{\sqrt{c}}$. Finally, (\ref{itm:e}) follows from
$\sqrt{|u|}=\sqrt{\sqrt{c^2+v^2}}\leq \sqrt{c+\sqrt{c^2+v^2}} = \Re \sqrt{2u}$.
\end{proof}

\end{document}